\documentclass[a4paper,10pt]{amsart}

\usepackage[utf8]{inputenc}
\usepackage{amssymb}
\usepackage{amsmath}
\usepackage{amsthm}
\usepackage{color}
\usepackage{enumerate}
\usepackage{mathtools}

\newcommand{\mat}[4]{ \left ( \begin{array}{cc} #1 & #2 \\ #3 &
      #4 \end{array} \right)} 
\renewcommand{\vec}[2]{ \left ( \begin{array}{c} #1 \\ #2\end{array} \right)}

\newcommand{\fp}{{\rm fp}}

\newcommand{\rest}{\upharpoonright}

\newcommand{\Hom}{\operatorname{Hom}}

\newcommand{\End}{\operatorname{End}}

\newcommand{\Ext}{\operatorname{Ext}}

\newcommand{\Cogen}{\operatorname{Cogen}}
\newcommand{\Gen}{\operatorname{Gen}}
\newcommand{\Ker}{\operatorname{Ker}}

\newcommand{\tr}{\operatorname{tr}}
\newcommand{\rej}{\operatorname{rej}}
\newcommand{\Img}{\operatorname{Im}}
\newcommand{\Coim}{\operatorname{Coim}}
\newcommand{\Coker}{\operatorname{Coker}}

\DeclareMathOperator{\Soc}{Soc}
\DeclareMathOperator{\op}{op}
\DeclareMathOperator{\Ab}{Ab}

\newcommand{\Modr}[1]{\mathrm{Mod}\textrm{-}{#1}}

\newcommand{\Flatl}[1]{{#1}\textrm{-}\mathrm{Flat}}
\newcommand{\Modl}[1]{{#1}\textrm{-}\mathrm{Mod}}

\newcommand{\Prod}{\mathrm{Prod}}

\theoremstyle{plain}
\newtheorem{theoremsec}{Theorem}

\newtheorem{theorem}{Theorem}[section]
\newtheorem{lemma}[theorem]{Lemma}
\newtheorem{proposition}[theorem]{Proposition}
\newtheorem{corollary}[theorem]{Corollary}

\theoremstyle{definition}
\newtheorem{definition}[theorem]{Definition}

\theoremstyle{remark}
\newtheorem{remark}[theorem]{Remark}

\newtheorem{examples}[theorem]{Examples}

\theoremstyle{definition}

\theoremstyle{plain}

\usepackage{tikz}
\usepackage{tikz-cd}
\usepackage{hyperref}

\title{Flatness in finitely accessible additive categories}

\author{Manuel Cort\'es-Izurdiaga}
\address{Department of Mathematics, University of Almeria, E-04071, Almeria, Spain}
\email{mizurdia@ual.es}

\begin{document}

\begin{abstract}
Motivated by some problems proposed by Cuadra and Simson related to flat objects in finitely accessible Grothendieck categories, we study flatness in the more general setting of finitely accessible additive categories. For such category $\mathcal A$, we characterize when $\mathcal A$ is preabelian and abelian. We prove that if the class of flat objects in $\mathcal A$ is closed under pure subobjects, then every flat object is a direct union of \textit{small} flat subobjects. Finally, we characterize when $\mathcal A$ has enough flat and projective objects and we prove that, in this case, the class of flat objects is closed under pure subobjects.
\end{abstract}

\maketitle

\section{Introduction}

Flatness in algebra is a crucial property of modules and morphisms, playing a vital role in commutative  and non commutative algebra, and in algebraic geometry. Maybe, the notion of flatness is not the most intuitive one, but it is often more useful than other more intuitive notions, such us projectivity or freeness. As Mumford remarked in the beginning of Section III.10 of \cite{Mumford}:

\begin{quote}
	The concept of flatness is a riddle that comes out of algebra, but which
	technically is the answer to many prayers.
\end{quote}

From the point of view of relative homological algebra, flat modules can be used to compute resolutions and to define homological notions, such as the flat dimension of a module and the weak global dimension of the ring. The point here is, of course, that there are enough flat modules in module categories, i. e., every module is a quotient of a flat module. Sometimes, using flat modules instead of projectives have led to the solution of a problem. For instance, in order to prove that the global dimension of a ring is symmetric for two-sided noetherian rings, Auslander proved that the weak dimension is always symmetric and that, for such rings, the global weak and the left and right global dimensions coincide \cite[Corollary 8.28]{Rotman}.

One question arises naturally: can we apply the same ideas in an (abelian) category which does not have enough projectives but that has enough flats? The main points here are to find for which categories the notion of flat object makes sense, and if we are able to characterize those categories having enough flats.  This is the main idea behind the paper by Cuadra and Simson \cite{CuadraSimson}. Actually, they propose the following two questions relative to flat objects \cite[Open problems 2.9]{CuadraSimson}:

\begin{enumerate}
	\item Give a characterization of finitely accessible abelian categories having enough flat objects.
	
	\item If a finitely accessible abelian category has enough flat objects, does it have enough projective objects?
\end{enumerate}

In this paper, we give a characterization of those finitely accessible additive categories (with some extra condition, see Theorem \ref{t:FAACEnoughFlats}) having enough flat and projective objects, as well as some particular situations in which problem (2) has an affirmative answer, see Section \ref{s:Ring}.

The most categorical characterization of flatness relies on purity: A module $M$ is flat if every epimorphism ending in $M$ is pure. The natural setting to work with purity are the finitely accessible categories, that is, those which have direct limits, a set of isomorphism classes of finitely presented objects and such that every object in the category is a direct limit of finitely presented objects. Purity and flat objects in finitely accessible abelian categories were first studied by Stenström \cite{Stenstrom68}.

The main objective in this paper is to study flat objects in finitely accessible additive categories. The main tool is Crawley-Boevey's representation theorem \cite[Theorem of 1.4]{Crawley}, which states that every finitely accessible additive category $\mathcal A$ with class of finitely presented objects being $\fp(\mathcal A)$ is equivalent to the full subcategory of flat functors in the category $(\fp(\mathcal S)^{\textrm{op}}, \textrm{Ab})$ of contravariant additive functors from $\fp(\mathcal A)$ to the category $\Ab$ of abelian groups. Then, we fix a small preadditive category $\mathcal S$ and we study the full subcategory $\Flatl{\mathcal S}$ of flat functors in the category of all additive functors from $\mathcal S$ to the category of abelian groups.

In order to study the categorical properties of $\Flatl{\mathcal S}$, the torsion theory $\tau=(\mathcal T_\tau,\mathcal F_\tau)$ cogenerated by the class $\Flatl{\mathcal S}$ plays a crucial role. For this reason, we extend, in Section \ref{s:TorsionTheories}, many results known for torsion theories in module categories to our setting of functor categories.

Although our results are stated for the category of flat functors in a functor category, they have a direct translation into the setting of finitely accessible additive categories as a direct application of the aforementioned representation theorem. For instance, Theorem \ref{t:Preabelian} reads:

\begin{theoremsec}\label{t:FAACAbelian}
	Let $\mathcal A$ be a finitely accessible additive category. Then, the following are equivalent:
	\begin{enumerate}
		\item $\mathcal A$ is preabelian.
		
		\item $\mathcal A$ has cokernels.
		
		\item $\mathcal A$ is a reflective subcategory of $(\fp(\mathcal A)^{\op},\Ab)$.
		
		\item $\mathcal A$ is preenveloping and closed under cokernels in $(\fp(\mathcal A)^{\op},\Ab)$.
		
		\item $(\fp(\mathcal A),\Ab)$ is locally coherent and $(\fp(\mathcal A)^{\op},\Ab)$ has weak dimension less than or equal to $2$.
	\end{enumerate}
	Moreover, $\mathcal A$ is abelian if and only if $(\fp(\mathcal A),\Ab)$ is locally coherent and for every object $M$ in $\mathcal A$ and subobject $K$ of $M$ in $(\fp(\mathcal A)^{\op},\Ab)$, the inclusion $K \hookrightarrow \rej_{\mathcal A}^K(M)$ is an $\mathcal A$-reflection of $K$.
\end{theoremsec}

In this theorem we see the two types of results we usually obtain: Assertions (1) and (2) state an internal characterization of $\mathcal A$ (the remarkable fact that if $\mathcal A$ has cokernels then it has kernels). Assertions (3), (4) and (5) are related to how $\mathcal A$ sits inside the functor category $(\fp(\mathcal A)^{\op},\Ab)$.

Related to flat objects we can prove the following, which, as a consequence of Corollary \ref{c:EnoughFlatsImpliesPureSubmodules}, is an extension of \cite[Theorem 3]{Rump} (here, $\tau$ is the torsion theory in $(\fp(\mathcal A)^{\op}, \Ab)$ cogenerated by - the subcategory equivalent to - $\mathcal A$):

\begin{theoremsec}\label{t:FAACFlats}
	Let $\mathcal A$ be a finitely accessible additive category such that the torsion theory $\tau$ is hereditary and $\mathcal A$ is closed under pure subobjects in $(\fp(\mathcal A)^{\op},\Ab)$. Then there is a cardinal $\kappa$ such that:
	\begin{enumerate}
		\item Every flat object in $\mathcal A$ is the direct union of $<\kappa$-presented (in $\Flatl{\mathcal S})$ flat objects in $\mathcal A$.
		
		\item Every flat object in $\mathcal A$ is filtered in the pure-exact structure by $<\kappa$-presented (in $\Modl{\mathcal S}$) flat objects.
	\end{enumerate}
\end{theoremsec}

Finally, we characterize those finitely accessible additive categories $\mathcal A$ having enough flats and projectives (theorems \ref{t:ExistenceEnoughFlats} and \ref{t:ExistenceEnoughProjectives}):

\begin{theoremsec}\label{t:FAACEnoughFlats}
	Let $\mathcal A$ be a finitely accessible additive category and suppose that $\tau$ is hereditary. The following assertions are equivalent:
	\begin{enumerate}
		\item There are enough flats (resp. projective) objects in $\mathcal A$.
		
		\item The torsion theory $\tau$ is jansian and for every representable functor $H_a$ of $(\fp(\mathcal A)^{\op},\Ab)$ there exists an epimorphism $f^a:G_a \rightarrow C_\tau(H_a)$ in $(\fp(\mathcal A)^{\op},\Ab)$ with $G_a$ being a $\tau$-cotorsion-free flat (resp. projective) module in $(\fp(\mathcal A)^{\op}, \Ab)$.
		
		\item There exists a pseudoprojective, $\tau$-cotorsion-free and flat (resp. projective) in $(\fp(\mathcal A)^{\op},\Ab)$ module $P$ such that $\mathcal T_\tau=\{X \in (\fp(\mathcal A)^{\op},\Ab) \mid \Hom(P,X)=0\}$.
	\end{enumerate}
	Moreover, when there are enough flats in $\mathcal A$, the class of flat objects in $\mathcal A$ is closed under pure subobjects.
\end{theoremsec}

\section{Functor categories}
\label{sec:preliminaries}

In this section, we state the notation and terminology that we use in the paper. We denote by $|X|$ the cardinality of a set $X$. The next cardinal of a cardinal $\kappa$ is denoted by $\kappa^+$. The cardinal $\kappa$ is said to be \textit{regular} if it is not the union of a family of less that $\kappa$ sets with cardinality smaller than $\kappa$.

A category $\mathcal B$ is \textit{preadditive} if for each pair of objects $A, B$ of $\mathcal B$, the set $\Hom_{\mathcal B}(A,B)$ is equipped with an abelian group structure such that the composition is bilinear.  Let $f:A \rightarrow B$ be a morphism in $\mathcal B$. If $A'$ is a subobject of $A$, the restriction of $f$ to $A'$ is denoted $f\rest A'$. A \textit{pseudocokernel} or a \textit{weak cokernel} of $f$ is a morphism $c:B \rightarrow C$ such that $cf=0$ and satisfying that for any morphism $d:B \rightarrow C'$ with $df=0$, there exists $g:C \rightarrow C'$ (not necessarily unique) satisfying $gc=d$. Dually, they are defined \textit{pseudokernels} or \textit{weak kernels}.

Fixed a class of objects $\mathcal C$ of $\mathcal B$, a morphism $f:B \rightarrow C$ with $C \in \mathcal C$ is a \textit{$\mathcal C$-preenvelope} if $\Hom_{\mathcal B}(f,C')$ is monic for every $C' \in \mathcal C$; $f$ is a \textit{$\mathcal C$-reflection} if $\Hom(f,C')$ is an isomorphism for every $C' \in \mathcal C$. Every object has a $\mathcal C$-reflection if and only if $\mathcal C$ is a \textit{reflective subcategory} of $\mathcal B$, that is, the inclusion functor from $\mathcal C$ to $\mathcal B$ has a left adjoint. Dually, they are defined $\mathcal C$-precovers, $\mathcal C$-coreflections and coreflective subcategories. If $\mathcal B$ is an abelian category and $\mathcal X$ is the class of all injective objects, then $\mathcal X$-envelopes coincide with the classical notion of injective envelope or injective hull. We will denote by $E(M)$ the injective envelope of the object $M$.

The preadditive category $\mathcal B$ is \textit{additive} if it has a zero object and finite products. An \textit{additive subcategory} of $\mathcal B $ is a full subcategory $\mathcal C$ of $\mathcal B$ containing the zero object and being closed under finite products. Recall that an additive category $\mathcal C$ is \textit{preabelian} if it has kernels and cokernels. 

An \textit{exact category} is an additive category $\mathcal B$ equipped with a class of kernel-cokernel pairs $\mathcal E$, called short exact sequences, that satisfies the axioms [E0], [E1] and [E2] (and their opposites) of \cite[Definition 2.1]{Buhler}. The typical example of an exact category is the class of all short exact sequences in an abelian category. If $\mathcal X$ is a class of objects of $\mathcal B$, an object $B$ of $\mathcal B$ is $\mathcal X$-filtered if there exists a family of subobjects of $B$ (called the $\mathcal X$-filtration), $(B_\alpha \mid \alpha < \mu)$, indexed by some ordinal $\mu$, such that $B=\bigcup_{\alpha < \mu}B_\alpha$, $B_0=0$ and 
\begin{displaymath}
	\begin{tikzcd}
		0 \arrow{r} & B_\alpha \arrow[hook]{r}{u_\alpha} & B_{\alpha+1} \arrow{r} & \Coker u_\alpha \arrow{r} & 0
	\end{tikzcd}
\end{displaymath}
is short exact with $\Coker u_\alpha \in \mathcal X$ for every $\alpha < \kappa$. A class $\mathcal C$ of $\mathcal B$ is \textit{deconstructible} if there exists a set of objects $\mathcal X \subseteq \mathcal C$ such that $\mathcal C$ consists of all $\mathcal X$-filtered modules.

An object $B$ of the additive category $\mathcal B$ is \textit{finitely presented} if the functor $\Hom_{\mathcal B}(B,-)$ commutes with direct limits. The additive category $\mathcal B$ is \textit{finitely accessible} if it has direct limits, if every object is a direct limit of finitely presented objects and if the full subcategory, $\fp(\mathcal B)$, of finitely presented objects of $\mathcal B$ is skeletally small.

Let $\mathcal S$ be a small preadditive category that we fix through the paper. We use under-case letters for objects and morphisms in $\mathcal S$. A \textit{left $\mathcal S$-module} is an additive functor $M$ from $\mathcal S$ to the category of abelian groups. Morphisms between left $\mathcal S$-modules $M$ and $N$ are just natural transformations $f:M \rightarrow N$; we will denote by $f_a$ the group morphism from $M(a)$ to $N(a)$ for each $a \in \mathcal S$. Left $\mathcal S$-modules and morphisms form a Grothendieck category, which we denote by $\Modl{\mathcal S}$, having the family of \textit{representable functors}, $\{\Hom_{\mathcal S}(a,-)\mid a \in \mathcal S\}$, as a set of finitely generated projective generators. We denote each representable functor $\Hom_{\mathcal S}(a,-)$ by $H_a$. Given two left $\mathcal S$-modules $M$ and $N$, the unadorned $\Hom(M,N)$ stands for the set of morphisms $\Hom_{\Modl{\mathcal S}}(M,N)$ between $M$ and $N$ in $\Modl{\mathcal S}$. Dually, $\Modr{\mathcal S}$ is the functor category consisting of all additive functors from the opposite category of $\mathcal S$, $\mathcal S^{\textrm{op}}$, to the category of abelian groups. Throughout the paper, module means left $\mathcal S$-module.

This notation and terminology is based on the fact that if $\mathcal S$ consists of one object $a$, then $\Modl{\mathcal S}$ is equivalent to the module category $\Modl{\mathcal \End_{\mathcal S}(a)}$. For this reason, modules (over non commutative rings) are our main source of inspiration. Moreover, they serve to illustrate the results on functor categories in a more concrete way, see Section \ref{s:Ring}. Through the paper, $R$ will be an associative ring with unit and we denote by $\Hom_R(M,N)$ the set of morphisms from $M$ to $N$ for every pair of left $R$-modules.


Many of the notions in the Grothendieck category $\Modl{\mathcal S}$ can be interpreted "locally". For instance, if $M$ is a module and $F,G$ are submodules of $M$, the sum $F+G$ and the intersection $F \cap G$ computed as, for instance, in \cite[IV.3]{Stenstrom}, are the functors that satisfy $(F+G)(a)=F(a)+G(a)$ and $(F \cap G)(a) = F(a) \cap G(a)$ for each $a \in \mathcal S$. The kernel (resp. image) of a morphism in $\Modl{\mathcal S}$, $f:M \rightarrow N$, is the submodule $\Ker f$ of $M$ (resp. $\Img f$ of $N$) that satisfies $(\Ker f)(a)=\Ker f_a$ (resp. $(\Img f)(a)=\Img f_a$) for each $a \in \mathcal S$.

One useful fact is the following characterization of submodules of a module. A submodule of a module $M$ is determined by a family of abelian groups, $\{K_a \mid a \in \mathcal S\}$, satisfying that $K_a$ is a subgroup of $M(a)$ with the additional property that for any morphism $r:a \rightarrow b$ in $\mathcal S$, $M(r)(K_a) \leq K_b$. For such family of groups, the functor $K$ defined by $K(a)=K_a$ for every $a \in \mathcal S$ and $K(r) = M(r) \rest K_a$ for every morphism $r:a \rightarrow b$ in $\mathcal S$, is a submodule of $M$. We call it the \textit{submodule induced by the family of subgroups $\{K_a \mid a \in \mathcal S\}$}.


An element $x$ of $M$ is an element of $\bigcup_{a \in \mathcal S}M(a)$; we use the notation $x \in M$ and, if $x \in M(a)$ for some $a$, we say that $x$ is an $a$-element. A subset $X$ of $M$ is a subset of $\bigcup_{a \in \mathcal S}M(a)$, for which we use the notation $X \subseteq M$. The \textit{cardinality of $X$} is $|X|=\sum_{a \in \mathcal S}|X(a)|$, where $X(a)=\{x \in X \mid x \in M(a)\}$ for each $a \in \mathcal S$. The \textit{weight of $\mathcal S$} \cite{Simson77} is the cardinal number $w(\mathcal S)= \sum_{a,b \in \mathcal S}|\Hom_{\mathcal S}(a,b)|$.

Given $a \in \mathcal S$, a subset $X \subseteq M(a)$ and a submodule $I$ of $H_a$, we denote by $IX$ the submodule of $M$ induced by the family of subgroups $\left\{IX(b)\mid b \in \mathcal S\right\}$, where
\begin{displaymath}
	IX(b)=\left\{\sum_{i=1}^nM(r_i)(x_i)\mid r_i \in I(b), x_i \in X\right\}
\end{displaymath}
for every $b \in \mathcal S$. In particular, for $I=H_a$, we obtain \textit{the submodule generated by} $X$, $H_aX$, which we denote by $\langle X \rangle$ as well. In \cite[Lemma 10.1.12]{Prest} it is essentially proved that the module $H_a$ is equal to $\langle 1_a\rangle$ for each $a \in \mathcal S$, where $1_a$ denotes the identity of $a$ in $\mathcal S$. Moreover, given $x \in M(a)$, it is very easy to see that there is a morphism  $m^x:H_a \rightarrow M$ given by the rule $m^x_b(r) = M(r)(x)$ for any $r \in H_a(b)$, whose image is $\langle x\rangle$.

Let $M$ be a module and $L$, a submodule of $M$. For any $a$-element $x$ of $M$ ($a \in \mathcal S$), it is easy to verify that the family of abelian groups $\{(x:L)(b) \mid b \in \mathcal S\}$, where
\begin{displaymath}
(x:L)(b) = \{r \in H_a(b) \mid M(r)(x) \in L(b)\},
\end{displaymath}
induces a submodule $(x:L)$ of $H_a$. This submodule is, precisely, the kernel of the morphisms $m^{x+L}:H_a \rightarrow M/L$ (where $x+L$ is the $a$-element $x+L(a)$ of $M/L$), so that $H_a/(x:L) \cong \langle x+L\rangle$.

Let $\kappa$ be a cardinal. A left $\mathcal S$-module $M$ is $ \kappa$-generated (resp. $\kappa$-presented) if there exists a family of representable functors, $\{S_i\mid i \in I\}$, (resp. families of representable functors $\{S_i\mid i \in I\}$ and $\{T_j \mid j \in J\}$) with $|I|=\kappa$ (resp. $|I|$ and $|J|$ having cardinality $\kappa$) and an epimorphism $\bigoplus_{i \in I}S_i \rightarrow M$ (resp. and an exact sequence $\bigoplus_{j \in J}T_j \rightarrow \bigoplus_{i \in I}S_i \rightarrow M \rightarrow 0$). If $\kappa$ is finite we just say that $M$ is finitely generated (resp. finitely presented). The following result is well known for modules over rings.

\begin{proposition}\label{p:KappaPresented}
	Let $\kappa$ be a cardinal and $M$, a left $\mathcal S$-module. 
	\begin{enumerate}
		\item $M$ is $ \kappa$-generated if and only if there exists a subset of cardinality $\kappa$, $X \subseteq M$, such that $M=\langle X \rangle$.
		
		\item If $\kappa$ is infinite regular and strictly greater than $w(\mathcal S)$, then $M$ is $<\kappa$-generated if and only if $M$ is $<\kappa$-presented if and only if $|M|<\kappa$.
	\end{enumerate}
\end{proposition}

\begin{proof}
	(1) Straightforward.
	
	\medskip
	
	(2) Simply notice that if $M=\langle X \rangle$ for some subset $X$ of $M$ with cardinality smaller than $\kappa$, then $\langle x \rangle(a)$ has cardinality smaller than $\kappa$ and so does $M(a)$, since $M(a)=\sum_{x \in X}\langle x \rangle(a)$. Then $|M| = |\bigcup_{a \in \mathcal S}M(a)|$ is smaller than $\kappa$.
\end{proof}

Now we recall some facts concerning the reject and the trace with respect to a class of modules. Given a class $\mathcal X$ of modules and $M$, a module,  we denote by $\rej_{\mathcal X}(M)$ the \textit{reject of $\mathcal X$ in $M$}, that is, the intersection of the kernels of all morphisms $f:M \rightarrow X$ with $X \in \mathcal X$. The reject of $\mathcal X$ is a radical (i. e., $\rej_{\mathcal X}(M/\rej_{\mathcal X}(M))=0$ for each module $M$) and, if we denote $\rej_{\mathcal X}\big(\rej_{\mathcal X}(M)\big)$ by $\rej_{\mathcal X}^2(M)$, then $\rej_{\mathcal X}(M)/\rej^2_{\mathcal X}(M)$ is the reject of $\mathcal X$ in $M/\rej_{\mathcal X}(M)$. If $K$ is a submodule of $M$, we denote by $\rej_{\mathcal X}^K(M)$ the intersection of the kernels of all morphisms $f:M \rightarrow X$ satisfying that $X \in \mathcal X$ and $K \leq \Ker f$. Clearly, it is satisfied $\rej_{\mathcal X}(M/K) = \rej_{\mathcal X}^K(M)/K$.
It is easy to see that the kernel of every $\mathcal X$-preenvelope $f:M \rightarrow X$ is precisely $\rej_{\mathcal X}(M)$. The \textit{class of modules cogenerated by $\mathcal X$}, $\Cogen(\mathcal X)$, consists of all modules $X$ for which there exists a monomorphism from $X$ to a direct product of modules belonging to $\mathcal X$. Clearly, $X \in \Cogen(\mathcal X)$ if and only if $\rej_{\mathcal X}(X)=0$.

Dually, the \textit{trace of $\mathcal X$} in the module $M$ is the submodule $\tr_{\mathcal X}(M)$ which is equal to the sum of the images of all morphisms $f:X \rightarrow M$ with $X \in \mathcal X$. The \textit{class of modules generated by $\mathcal X$}, $\Gen(\mathcal X)$, consists of all modules which are a quotient of a direct sum of modules belonging to $\mathcal X$. Clearly, $M \in \Gen(\mathcal X)$ if and only if $\tr_{\mathcal X}(M)=M$. The following result, which is well known in module categories, state some properties of the trace in our context:

\begin{lemma}\label{l:GenMClosedExtensions}
	Let $M$ be a module:
	\begin{enumerate}
		\item For every module $N$ and $a \in \mathcal S$, $\tr_M(H_a)\cdot N(a) \leq \tr_M(N)$.
		
		\item $\Gen(M)$ is closed under extensions if and only if $\tr_M(X/tr_M(X))=0$ for every module $X$ (i. e., $\tr_M$ is a radical).
	\end{enumerate}
\end{lemma}

\begin{proof}
	(1) Fix $b \in \mathcal S$ and take $r \in \tr_M(H_a)(b)$ and $x \in N(a)$. We prove that $N(r)(x) \in \tr_M(N)(b)$. Using that $r \in \tr_M(H_a)(b)$ we can find $f^1, \ldots, f^n \in \Hom(M,H_a)$ and $x_1, \ldots, x_n \in M(b)$ such that $r=\sum_{i=1}^nf^i_b(x_i)$. Then
	\begin{displaymath}
		N(r)(x) = \sum_{i=1}^nN(f^i_b(x_i))(x) = \sum_{i=1}^n (m^x \circ f^i)_b(x_i),
	\end{displaymath}
	which means that $N(r)(x) \in \tr_M(N)(b)$.
	
	(2). $\Rightarrow$. Given a module $X$ and $f:M \rightarrow X/\tr_M(X)$, we can compute the pullback of $f$ along the projection $p:X \rightarrow X/\tr_M(X)$ to get the following commutative diagram with exact rows
	\begin{displaymath}
		\begin{tikzcd}
			0 \arrow{r} & \tr_M(X) \arrow{r} \arrow[equal]{d} & Q \arrow{r} \arrow{d}{\overline f}  & M \arrow{r} \arrow{d}{f}& 0\\
			0 \arrow{r} & \tr_M(X) \arrow[hook]{r} & X \arrow{r}{p} & X/\tr_M(X) \arrow{r} & 0\\
		\end{tikzcd}
	\end{displaymath}
	Since $\Gen(M)$ is closed under extensions, $Q \in \Gen(M)$. Then $\Img(\overline f) \leq \tr_M(X)$ and $p\overline f=0$. This implies that $f=0$.
	
	$\Leftarrow$. Take a short exact sequence
	\begin{displaymath}
		\begin{tikzcd}
			0 \arrow{r} & A \arrow{r} & B \arrow{r} & C \arrow{r} & 0
		\end{tikzcd}
	\end{displaymath}
	with $A,C \in \Gen(M)$. Then, $A \leq \tr_M(B)$, so that $B/\tr_M(B)$ is isomorphic to a quotient of $C$, which means that it belongs to $\Gen(M)$. Using (2), this implies that $B/\tr_M(B)=0$.
\end{proof}

We finish this section discussing the pure exact structure in $\Modl{\mathcal S}$. A short exact sequence in $\Modl{\mathcal{S}}$ is \textit{pure} \cite{Stenstrom68} if every finitely presented module is projective with respect to it. Using the well known properties of finitely presented modules in $\Modl{\mathcal{S}}$ \cite{Stenstrom68}, the proof of \cite[34.5]{Wisbauer} can be translated to this setting, so that a short exact sequence
\begin{displaymath}
	\begin{tikzcd}
		0 \arrow{r} & M \arrow{r} & N \arrow{r} & L \arrow{r} & 0
	\end{tikzcd}
\end{displaymath}
is pure if and only if for any commutative diagram
\begin{displaymath}
	\begin{tikzcd}
		& F \arrow{d}{g} \arrow{r}{f} & G \arrow{d} & & \\
		0 \arrow{r} & M \arrow{r} & N \arrow{r} & L \arrow{r} & 0
	\end{tikzcd}
\end{displaymath}
with $F$ and $G$ finitely presented, there exists $h:G \rightarrow M$ with $hf=g$ (see \cite[Lemma 4.4]{Stovicek14}). We now extend the so-called purification lemma \cite[Lemma 5.3.12]{EnochsJenda} to the functor category $\Modl{\mathcal S}$.

\begin{theorem}\label{t:Purification}
	\begin{enumerate}
		\item Let  $K$ be a submodule of a module $M$. Then, there exists a pure submodule $\overline K$ of $M$ containing $K$ with cardinality less than or equal to $\max\{w(\mathcal S),|K|\}$.
		
		\item Every module is a direct union of $<w(\mathcal S)^+
		$-presented pure submodules.
	\end{enumerate}
\end{theorem}

\begin{proof}
	(1) Set $\kappa=\max\{w(\mathcal S),|K|\}$. We follow the proof of \cite[Lemma 5.3.13]{EnochsJenda}. We are going to construct by induction a chain of submodules $K=K_0 \leq K_1 \leq K_2 \leq \cdots$ of $M$ with $|K_n|\leq \kappa$ and satisfying that $(\star)$ for every commutative diagram
	\begin{displaymath}
		\begin{tikzcd}
			& F \arrow{d}{f} \arrow{r}{g} & G \arrow{d} \\
			0 \arrow{r} & K_n \arrow{r} & M \\
		\end{tikzcd}
	\end{displaymath}
	there exists $h:G \rightarrow K_{n+1}$ satisfying $hg=f$.
	
	For $n=0$, set $K_0=K$.
	
	Suppose that we have constructed $K_n$ for some $n < \omega$. Consider the following set:
	\begin{displaymath}
		\Lambda=\{(f,g,F,G) \mid f:F \rightarrow K_n, \, g:F \rightarrow G \textrm{ are morphisms with f. pres. }F, G \}
	\end{displaymath}
	Then, it is easy to see that $|\Lambda| \leq \kappa$. Consider the subset $\Lambda'$ of $\Lambda$ consisting of all tuples $(f,g,F,G)$ in $\Lambda$ for which there exists some $h:G \rightarrow M$ satisfying $hg=f\rest K_n$. For any tuple $\lambda=(f,g,F,G)$ in $\Lambda'$, fix one such $h_\lambda:G \rightarrow M$ and define
	\begin{displaymath}
		K_{n+1}=K_n+\sum_{\lambda \in \Lambda'}\Img h_\lambda.
	\end{displaymath}
	Since every finitely presented module $H$ satisfies that $|H| \leq w(\mathcal S)$, $|K_{n+1}|\leq \kappa$. Moreover, $K_{n+1}$ clearly satisfies $(\star)$. This concludes the construction.
	
	Finally, setting $\overline K=\bigcup_{n < \omega}K_n$ be obtain the desired pure submodule of $M$.
	
	(2) Every module $M$ is the direct union of its $<w(\mathcal S)^+$-presented submodules. But, by (1) and Proposition \ref{p:KappaPresented}, every the set of $<w(\mathcal S)^+$-presented pure submodules is cofinal in the direct set of all $<w(\mathcal S)^+$-presented submodules.
\end{proof}

Our definition agrees with the usual notion of purity in non-additive accessible categories \cite[Definition 2.27]{AdamekRosicky} (just notice that, for every infinite regular cardinal $\lambda$, the $\lambda$-presentable objects in \cite{AdamekRosicky} are our $<\lambda$-presented). There is a purification lemma for these categories, \cite[Theorem 2.33]{AdamekRosicky}. In Theorem \ref{t:Purification}, we obtain more information about the cardinal of the submodule $\overline K$. On the other hand, notice that (2), is precisely \cite[Corollary 2]{Rump}, since the small modules in \cite{Rump} are, precisely, the $<w(\mathcal S)^+$-presented modules by Proposition \ref{p:KappaPresented}.

A module $M$ is \textit{flat} if every epimorphism $f:X \rightarrow M$ is pure. Flat modules in $\Modl{\mathcal S}$ satisfy the same properties than flat modules in $\Modl R$, see \cite[49.5]{Wisbauer},  \cite[p. 1646]{Crawley}, \cite{OberstRohrl} and \cite{Stenstrom}. We denote by $\Flatl{\mathcal S}$ the class (and the full subcategory) consisting of all flat modules. The category $\Flatl{\mathcal S}$ is a finitely accessible additive category \cite[Theorem 1.4]{Crawley}, in which direct limits are computed in $\Modl{\mathcal S}$ and the finitely presented modules are the finitely generated projective modules in $\Modl{\mathcal S}$. As a consequence of Theorem \ref{t:Purification} we get:

\begin{theorem}\label{t:Deconstructible}
	\begin{enumerate}
		\item Every flat module is the direct union of $<w(\mathcal S)^+$-presented flat modules.
		
		\item Every flat module is filtered by $w(\mathcal S)^+$-presented flat modules. In particular, $\Flatl{\mathcal S}$ is deconstructible.
	\end{enumerate}
\end{theorem}

\begin{proof}
	(1) Follows from Theorem \ref{t:Purification}(2) and the fact that pure submodules of flat modules are flat.
	
	\medskip
	
	(2) Given a module $F \in \Flatl{\mathcal S}$, we may assume that $F$ is $\mu$-presented for some $\mu \geq w(\mathcal S)^+$. Write $\bigcup_{a \in \mathcal S}F(a)=\{x_\alpha \mid \alpha < \mu\}$. Then, the filtration of $F$, $\{F_\alpha \mid \alpha < \kappa\}$, satisfying that $F_\alpha$ is pure in $F$ and contains $\{x_\gamma \mid \gamma < \alpha\}$ can be constructed using a standard deconstruction procedure (by transfinite recursion), that we sketch here for completeness:
	\begin{itemize}
		\item For $\alpha=0$, take $F_0=0$.
		
		\item If for some $\alpha < \kappa$ we have just constructed $F_\alpha$, consider the quotient $F/F_\alpha$ and apply Theorem \ref{t:Purification} to get a $<w(\mathcal S)^+$-presentable pure submodule $F_{\alpha+1}/F_\alpha$ of $F/F_\alpha$ containing $x_\alpha+F_\alpha$. Since $F_\alpha$ is pure in $F$, $F/F_\alpha$ is flat and so is $F_{\alpha+1}/F_\alpha$. Moreover, $F_{\alpha+1}$ is pure in $F$ and contains $x_\alpha$.
		
		\item If $\alpha$ is a limit ordinal such that we have constructed $F_\gamma$ for every $\gamma<\alpha$, just take $F_\alpha$ to be the union $\bigcup_{\gamma < \alpha}F_\gamma$.
	\end{itemize}
\end{proof}

We finish this section studying the existence of flat preenvelopes in $\Modl{\mathcal S}$.

\begin{lemma}\label{l:CogFlats}
	Let $\kappa>w(\mathcal S)$ be an infinite regular cardinal and $M$ $<\kappa$-presented module. Let $\mathcal H$ be a set of representatives of the isomorphism classes of all $<\kappa$-presented flat modules and $\varphi_M$, the canonical morphism from $M$ to the product $\prod_{H \in \mathcal H}H^{\Hom(M,H)}$. Then:
	\begin{enumerate}
		\item Every morphism  $f:M \rightarrow \prod_{i \in I}H_i$, where $\{H_i\mid i \in I\}$ is a family of flat modules, factors through $\varphi_M$. In other words, $\varphi_M$ is a $\Prod(\Flatl{\mathcal S})$-preenvelope, where $\Prod(\Flatl{\mathcal S})$ denotes the class of all direct products of flat modules.
		
		\item $\Ker \varphi_M = \rej_{\Flatl{\mathcal S}}(M)$.
		
		\item $M \in \Cogen(\Flatl{\mathcal S})$ if and only if $\rej_{\Flatl{\mathcal S}}(M)=0$.
	\end{enumerate}
\end{lemma}

\begin{proof}
	(1) Denote by $p_i:\prod_{i \in I}H_i \rightarrow H_i$ the canonical projection and by $H_M$ the product $\prod_{H \in \mathcal H}H^{\Hom(M,H)}$. Since $\kappa>w(\mathcal S)$, $\Img p_if$ is $<\kappa$-presented by Proposition \ref{p:KappaPresented}, so that there exists, in view of Theorem \ref{t:Purification}, a pure submodule $\overline K_i \leq H_i$ containing $\Img p_if$ and with $|\overline K_i|<\kappa$. Since $H_i$ is flat, $\overline K_i$ is flat as well \cite[49.5]{Wisbauer}. Then, there exists $g_i:H_M \rightarrow H_i$ with $g_i\varphi_M=p_if$. Then the induced morphisms $g:H_M \rightarrow \prod_{i \in I}F_i$ by these $g_i's$ satisfy that $g\varphi_M=f$.
	
	(2) and (3) are straightforward by (1).
\end{proof}

In general, direct products of flat modules are not flat. This condition is satisfied when the category $\Modr{\mathcal S}$ is locally coherent, i. e., every finitely presented right $\mathcal S$-module is coherent \cite[Theorem 4.1]{OberstRohrl}. As a consequence of the preceding result, we obtain the following extension of \cite[Proposition 6.5.2]{EnochsJenda} to functor categories:

\begin{corollary}\label{c:CoherentCategories}
	The module category $\Modr{\mathcal S}$ is locally coherent if and only if every left $\mathcal S$-module has a flat preenvelope.
\end{corollary}

\begin{proof}
	If $\Modr{\mathcal S}$ is locally coherent, the morphism $\varphi_M$ defined in Lemma \ref{l:CogFlats} is a flat preenvelope of $M$ for any module $M$. The converse follows from the fact that any preenveloping class closed under direct summands is closed under direct products as well.
\end{proof}

The flat dimension of modules and the weak dimension of $\Modl{\mathcal S}$ can be defined as in the case of rings, see \cite[Section 50]{Wisbauer}.


\section{Torsion theories in functor categories} \label{s:TorsionTheories}

A \textit{torsion theory} in $\Modl{\mathcal S}$ is a pair of classes of modules, $\sigma=(\mathcal{T}_\sigma,\mathcal{F}_\sigma)$, such that $T \in \mathcal T_\sigma$ if and only if $\Hom(T,X)=0$ for each $X \in \mathcal F_\sigma$, and $F \in \mathcal F_\sigma$ if and only if $\Hom(X,F) =0$ for each $X \in \mathcal T_\sigma$. The elements in $\mathcal T_\sigma$ are called \textit{$\sigma$-torsion}, the elements of $\mathcal F_\sigma$, \textit{$\sigma$-torsion-free}, and we denote the corresponding idempotent radical (i. e., the trace with respect to $\mathcal T_\sigma$) by $\sigma$ as well. For undefined notions and results about torsion theories we refer to \cite[Chapter VI]{Stenstrom}, \cite[Section 11]{Prest} and \cite{Golan}.

A submodule $K$ of a module $M$ is \textit{$\sigma$-dense} if $M/K$ is $\sigma$-torsion. We denote by $\mathcal U_\sigma(M)$ the set of all $\sigma$-dense submodules of $M$. If the torsion theory is hereditary, i. e., $\mathcal T$ is closed under submodules (equivalently, $\mathcal F$ is closed under injective hulls \cite[Proposition VI.3.2]{Stenstrom}), then the set of $\sigma$-dense submodules of any module are determined by the family $\{\mathcal U_\sigma(H_a) \mid a \in \mathcal S\}$:

\begin{proposition}\label{p:DenseInducedByRepresentables}
Let $\sigma$ be a torsion theory in $\Modl{\mathcal S}$. Consider, for any module $M$, the family of submodules
\begin{displaymath}
\mathcal U'(M)=\{K \leq M \mid \forall a \in \mathcal S,\forall x \in M(a),(x:K) \in \mathcal U_\sigma(H_a)\}.
\end{displaymath}
Then:
\begin{enumerate}
	\item $\mathcal U'(M) \subseteq \mathcal U_\sigma(M)$ for every module $M$.
	
	\item $\sigma$ is hereditary if and only if $\mathcal U'(M) = \mathcal U_\sigma(M)$ for every module $M$.
\end{enumerate}
  
\end{proposition}

\begin{proof}
(1) Let $K \leq M$. If $x$ is an $a$-element of $M$, $m_{x+K}:H_a \rightarrow M/K$ is a morphism with image $\langle x+K\rangle$ and kernel $(x:K)$, where $x+K$ is the $a$-element $x+K(a)$ of $M/K$. Then, if $(x:K) \in \mathcal U_\sigma(H_a)$, $\langle x+K\rangle  \in \mathcal T_\sigma$. Since this happens for every element $x+K$ in $M/K$, there is an epimorphism $\bigoplus_{a \in \mathcal S}\bigoplus_{x \in M(a)}F_{a,x}\rightarrow M/K$, where $F_{a,x}=H_a/(x:K)$ for every $a \in \mathcal S$ and $x \in M$. Then, $M/K \in \mathcal T_\sigma$ by \cite[Proposition VI.2.1]{Stenstrom}

(2) If $\sigma$ is hereditary and $M/K \in \mathcal T_\sigma$, then $H_a/(x:K)\cong \langle x+K\rangle \in \mathcal T_\sigma$ and $(x:K) \in \mathcal U_\sigma(H_a)$, so that $\mathcal U'(M)$ and $\mathcal U_\sigma(M)$ are equal in this case. Conversely, if $\mathcal U'(M)=\mathcal U_\sigma(M)$ for every module $M$ and $M$ is a $\sigma$-torsion module, then $0 \in \mathcal U'(M)$, so that $(x:0) \in \mathcal U_{\sigma}(H_a)$ for every $a$-element $x$ of $M$. Then, $\langle x \rangle\cong H_a/(x:0) \in \mathcal T_\sigma$ which means that $\sigma$ is hereditary.
\end{proof}

If $\sigma$ is a hereditary torsion theory in the category of modules over the ring $R$, the preceding proposition says that $\sigma$ is determined by the family of left ideals $\mathcal U_\sigma(R)$, the \textit{Gabriel topology associated to $\sigma$} \cite[Theorem VI.5.1]{Stenstrom}. Extending this idea, a \textit{Gabriel topology} on $\mathcal S$ is a family of sets of submodules of the representable functors, $ \mathcal V=\{\mathcal V_a \mid a \in \mathcal S\}$ where $\mathcal V_a \subseteq H_a$, that satisfies, for every $a \in \mathcal S$, the following conditions:
\begin{enumerate}
\item[(T1)] If $U \in \mathcal V_a$, $L \leq H_a$ and $U \leq L$ then $L \in \mathcal V_a$.

\item[(T2)] If $U,V \in \mathcal V_a$, the $U \cap V \in \mathcal V_a$.

\item[(T3)] If $U \in \mathcal V_a$, $b \in \mathcal S$ and $x$ is a $b$-element of $H_a$, then $(x:U)\in \mathcal V_b$.

\item[(T4)] If $L \leq H_a$ satisfies that there exists $U \in \mathcal V_a$ such that for any $b \in \mathcal S$ and $b$-element $x \in U$, $(x:L) \in \mathcal V_b$, then $L \in \mathcal V_a$.
\end{enumerate}

The relationship between Gabriel topologies and hereditary torsion theories is the same as in the case of rings \cite[Theorem VI.5.1]{Stenstrom}:

\begin{proposition}\label{p:TopologiesAndTorsion}
There are mutual inverse bijective maps between the sets of all Gabriel topologies on $\mathcal S$ and of hereditary torsion theories of $\Modl{\mathcal S}$ given by:
\begin{itemize}
\item If $\sigma$ is a hereditary torsion theory of $\Modl{\mathcal S}$, the associated Gabriel topology is $\mathcal U_\sigma=\{\mathcal U_\sigma(H_a) \mid a \in \mathcal S\}$.

\item If $\mathcal V = \{\mathcal V_a \mid a \in \mathcal S\}$ is a Gabriel topology on $\mathcal S$, the associated hereditary torsion theory is the one whose class of torsion modules is
\begin{displaymath}
\mathcal T=\{X \in \Modl{\mathcal S} \mid \forall a \in \mathcal S, \forall x \in X(a), (x:0) \in \mathcal V_a\}.
\end{displaymath}
\end{itemize}
\end{proposition}

\begin{proof}
The proof follows the lines of \cite[Theorem VI.5.1]{Stenstrom}.

First, let $\sigma$ be a hereditary torsion theory and let us prove that $\mathcal U_\sigma$ is a Gabriel topology. Clearly, $\mathcal U_\sigma$ satisfies (T1). For (T2), given $a \in \mathcal S$, just notice that $H_a/(U \cap V)$ is isomorphic to a submodule of $(H_a/U) \oplus (H_a/V)$ for each $U,V \in H_a$. The condition (T3) follows from Proposition \ref{p:DenseInducedByRepresentables}. 

Finally, take $L \leq H_a$ and $U \in \mathcal U_\sigma(H_a)$ satisfying (T4). Then, reasoning as in the proof of Proposition \ref{p:DenseInducedByRepresentables}, we conclude that $(U+L)/L$ is $\sigma$-torsion. Then we can consider the short exact sequence,
\begin{displaymath}
0 \rightarrow (U+L)/L \rightarrow H_a/L \rightarrow H_a/(U+L) \rightarrow 0,
\end{displaymath}
where the modules in the left and in the right are $\sigma$-torsion (notice that the module $H_a/(U+L)$ is a quotient of $H_a/U$, which is $\sigma$-torsion). Then $H_a/L$ is $\sigma$-torsion and $L \in \mathcal U_\sigma(H_a)$.

Conversely, let us see that $\mathcal T$ is a hereditary torsion class. Clearly, $\mathcal T$ is closed under submodules; it is closed under quotients by (T1) and under direct sums by (T2). Let us see that $\mathcal T$ is closed under extensions. Let $K$ be a submodule of $M$ such that $K$ and $M/K$ belong to $\mathcal T$. Fix $a \in \mathcal S$ and $x \in M(a)$. Then, $(x+K:0) = (x:K)$ belongs to $\mathcal V_a$. Moreover, for every $b \in \mathcal S$ and $b$-element $y \in (x:K)$, $(y:(x:0)) = (M(y)(x):0)$ and, since $M(y)(x) \in K$, $(M(y)(x):0) \in \mathcal V_b$. By (T4), $(x:0) \in \mathcal V_a$ and $M \in \mathcal T$.
\end{proof}

The following easy fact will be used later:

\begin{lemma}\label{l:ProductIsDense}
	Let $\sigma$ be a hereditary torsion theory and, for every $a \in \mathcal S$, let $I_a \in \mathcal U_\sigma(H_a)$. Then, for every module $M$, $\sum_{a \in \mathcal S}I_a\cdot M(a) \in \mathcal U_\sigma (M)$.
\end{lemma}

\begin{proof}
	Just notice that, for any $b \in \mathcal S$ and $x \in M(b)$, the submodule $(x:\sum_{a \in \mathcal S}I_a\cdot M(a))$ belongs to $\mathcal U_\sigma(H_b)$, because it contains $I_b$. Then, $\sum_{a \in \mathcal S}I_a\cdot M(a) \in \mathcal U_\sigma(M)$ by Proposition \ref{p:DenseInducedByRepresentables}.
\end{proof}

A module $M$ is $\sigma$-torsion if and only if every submodule of $M$ is $\sigma$-dense. At the other extreme, they are those modules having no proper $\sigma$-dense submodules: the \textit{$\sigma$-cotorsion-free modules} \cite[Chapter 7]{Golan}. The following result states the main properties of $\sigma$-cotorsion-free modules:

\begin{proposition}\label{p:PropertiesCotorsionFree}
	The class of all $\sigma$-cotorsion-free modules is closed under quotients and extensions. If $f:M \rightarrow N$ is a superfluous epimorphism with $N$ a $\sigma$-cotorsion-free module, then $M$ is $\sigma$-cotorsion free. If $\sigma$ is hereditary, the class of $\sigma$-cotorsion-free modules is closed under direct sums.
\end{proposition}

\begin{proof}
	The proof of \cite[Proposition 7.2]{Golan} works for not necessarily hereditary torsion theories over our context of functor categories.
\end{proof}

If $\sigma$ is a hereditary torsion theory, then the class of all $\sigma$-cotorsion-free modules is a torsion class and, in particular, every module $M$ has a largest $\sigma$-cotorsion-free submodule $C_\sigma(M)$ (see \cite[Proposition 7.7]{Golan}). Clearly, $C_\sigma(M)$ is contained in every $\sigma$-dense submodule of $M$, so that $C_\sigma(M)\leq L_\sigma(M)$, where $L_\sigma(M)$ is the intersection of all $\sigma$-dense submodules of $M$. The following lemma, which will be used in the last theorem of this section, states when $C_\sigma(M)$ and $L_\sigma(M)$ are equal:

\begin{lemma}\label{l:L=C}
	The following assertions about the hereditary torsion theory $\sigma$ and the module $M$ are equivalent:
	\begin{enumerate}
		\item $L_\sigma(M)$ is $\sigma$-cotorsion-free, i. e., $C_\sigma(M)=L_\sigma(M)$.
		
		\item $L_\sigma(M)$ is dense in $M$.
	\end{enumerate}
\end{lemma}

\begin{proof}
	Given a $\sigma$-dense submodule $K$ of $L_\sigma(M)$, the short exact sequence
	\begin{displaymath}
		\begin{tikzcd}
			0 \arrow{r} & \frac{L_\sigma(M)}{K} \arrow{r} & \frac{M}{K} \arrow{r} & \frac{M}{L_\sigma(M)} \arrow{r} & 0
		\end{tikzcd}
	\end{displaymath}
	says that $L_\sigma(M)$ is $\sigma$-dense in $M$ if and only if $K$ is $\sigma$-dense in $M$ if and only if $K \leq L_\sigma(M)$.
\end{proof}

Now we discuss some facts about jansian torsion theories. Recall that the torsion theory $\sigma$ is \textit{jansian} (or a \textit{TTF-theory} in which case $\mathcal T_\sigma$ is called a \textit{TTF-class}) if the class of all $\sigma$-torsion modules is closed under direct products and submodules. This means \cite[Proposition VI.2.2]{Stenstrom} that there exists a class $\mathcal C_\sigma$ such that $(\mathcal C_\sigma, \mathcal T_\sigma)$ is a torsion theory. Jansian torsion theories can be characterized in terms of pseudoprojective modules. Recall that a module $P$ is \textit{pseudoprojective} if for every epimorphism $f:M \rightarrow N$ and non-zero morphism $g:P \rightarrow N$, there exist an endomorphism of $P$, $f':P \rightarrow P$ and $g':P \rightarrow M$ such that $0 \neq fg' = f''g$. The following characterization of pseudoprojective modules is an extension of \cite[3.9]{Wisbauer96} to our setting. This result is crucial to prove that when a finitely accessible additive category has enough flats, then the class of flat objects is closed under pure subobjects, Corollary \ref{c:EnoughFlatsImpliesPureSubmodules}.

\begin{proposition}\label{p:Pseudoprojective}
	The following are equivalent for a module $M$.
	\begin{enumerate}
		\item $M$ is pseudoprojective.
		
		\item $\Gen(M)$ is closed under extensions and $\mathcal F_M=\{X \in \Modl{\mathcal S} \mid \Hom(M,X)=0\}$ is closed under quotients. I. e., $(\Gen(M),\mathcal F_M)$ is a cohereditary torsion theory.
		
		\item For every module $N$, $\tr_M(N)=\sum_{b \in \mathcal S}\tr_M(H_b)\cdot N(b)$.
		
		\item $M=\sum_{b \in \mathcal S}\tr_M(H_b)\cdot M(b)$.
	\end{enumerate}
\end{proposition}

\begin{proof}
	(1) $\Rightarrow$ (2) is straightforward.
	
	(2) $\Rightarrow$ (1). Let $p:A \rightarrow B$ be an epimorphism and $f:M \rightarrow B$ a non-zero morphism. The pullback of $f$ along $p$ gives morphisms $\overline p:P \rightarrow M$ and $\overline f:P \rightarrow A$ such that $p\overline f = f \overline p$. We claim that there exists $g:M \rightarrow P$ with $f\overline p g \neq 0$. Then, $\overline p g$ and $\overline f g$ are the morphisms witnessing the pseudoprojectivity of $M$.
	
	Suppose that the claim is false. Then $\tr_M(P) \leq \Ker f\overline p$. This means that there exists an epimorphism from $P/\tr_M(P)$ to $P/\Ker f\overline p \cong \Img f$ and, by Lemma \ref{l:GenMClosedExtensions} and (2), this implies that $\Hom(M,\Img f)=0$, in particular that $f=0$. This is a contradiction.
	
	(1) $\Rightarrow$ (3). Suppose that (3) is false and that there exists a module $N$ satisfying $N\neq \sum_{b \in \mathcal S}\tr_M(H_b)\cdot N(b)$. Denote by $L$ the submodule $\sum_{b \in \mathcal S}\tr_M(H_b)\cdot N(b)$ of $N$ and take $\varphi:\bigoplus_{i \in I}H_{c_i} \rightarrow \tr_M(N)$ an epimorphism where, for every $j \in I$, $\varphi k_j=m^{x_j}$ for some element $x_j \in \tr_M(N)(c_j)$ ($k_j$ being the inclusion $H_{c_j} \hookrightarrow \bigoplus_{i \in I}H_{c_i}$). By Lemma \ref{l:GenMClosedExtensions}, $L \leq \tr_M(N)$ and, since the quotient $\tr_M(N)/L$ is a nonzero module in $\Gen(M)$, we can find a nonzero morphism $f:M \rightarrow \tr_M(N)/L$.
	
	Now notice that every pair of morphisms $s:M \rightarrow M$ and $h:M \rightarrow \bigoplus_{i \in I}H_{c_i}$ satisfying $p\varphi h=fs$, where $p:\tr_M(N)	\rightarrow \tr_M(N)/L$ is the projection, satisfies that $fs=0$, since every morphism $h:M \rightarrow \bigoplus_{i \in I}H_{c_i}$ satisfies that $\Img \varphi h \leq L$. This contradicts (1).
	
	(3) $\Rightarrow$ (4). Trivial.
	
	(4) $\Rightarrow$ (1). Let $p:A \rightarrow B$ be an epimorphism and $f:M \rightarrow B$ a nonzero morphism. Fix $c \in \mathcal S$ and $y_c \in M(c)$ with $f_c(y_c)\neq 0$. By (4) we can find $b_1, \cdots, b_n \in \mathcal S$, $r_1^i, \ldots, r_n^i \in \tr_M(H_{b_i})(c)$ and $x_1^i, \ldots, x_n^i \in M(b_i)$ such that $y_c = \sum_{i=1}^n\sum_{j =1}^{n_i}M(r_j^i)(x_j^i)$.
	
	Since $f_c(y_c)\neq 0$, there exists $i$ and $j$ such that $f_c(M(r_j^i)(x_j^i))\neq 0$. Since $r_j^i \in \tr_M(H_{b_i})(c)$, we can find $f^1, \ldots, f^l \in \Hom(M,H_{b_i})$ and $z_1, \ldots, z_l \in M(c)$ such that $r_j^i=\sum_{j=1}^lf^j_c(z_j)$. In particular, there exists $k$ such that $f_c(M(f^k_c(z_k))(x_j^i))\neq 0$. Finally, using that $H_{b_i}$ is a projective module, we can find $h:H_{b_i}\rightarrow A$ satisfying $ph=f m^{x_j^i}$. Then taking $s=m^{x_j^i}f^k$ and $g=hf^k$, we get that $pg=fs$ and that $fs \neq 0$, since
	\begin{displaymath}
		(fs)_c(z_k)=f_c\big(M(f^k_c(z_k))(x_j^i)\big) \neq 0.
	\end{displaymath}
	This proves that $M$ is pseudoprojective.
\end{proof}

\begin{remark}\label{r:GolanMistake}
	In \cite[p. 20]{Golan} it is asserted that a left $R$-module $M$ is pseudoprojective if and only if the class $\{X \in \Modl{R} \mid \Hom_R(M,X)=0\}$ is a TTF class. This is not true unless $\Gen(M)$ is closed under extensions, as the preceding proposition suggests and the following example shows:
	
	Let $T$ be the ring of all the $\mathbb N$-square triangular matrices which are constant on the diagonal and have finitely many non-zero entries above the diagonal. Then, $T$ is local, right perfect and not left perfect. The left $T$-module $T/J$ is the unique simple left module (up to isomorphism) and, by Bass Theorem P \cite[Theorem 28.4]{AndersonFuller}, every non-zero left $T$-module contains a minimal submodule, i. e., $\{X \in \Modl{T} \mid \Hom_T(T/J,X)=0\}$ is the zero class. In particular, this class is a TTF-class, but $T/J$ is not a pseudoprojective left $T$-module as a consequence of \cite[Theorem 2.4]{CortesTorrecillas06}.
\end{remark}

The following characterization of jansian torsion theories is an extension of \cite[Proposition 4.17 and Corollary 4.22]{Golan} (see \cite{Wisbauer96} as well). It is crucial in proving our characterization of finitely accessible additive categories having enough flat objects and enough projective objects, Theorems \ref{t:ExistenceEnoughFlats} and \ref{t:ExistenceEnoughProjectives}.

\begin{proposition}\label{p:CharacterizationJansian}
	The following assertions are equivalent for the hereditary torsion theory $\sigma$.
	\begin{enumerate}
		\item $\sigma$ is jansian.
		
		\item $L_\sigma(M)$ is a $\sigma$-dense submodule of $M$ for every module $M$.
		
		\item $L_\sigma(H_a)$ is a $\sigma$-dense submodule of $H_a$ for every $a \in \mathcal S$.
		
		\item A module $M$ is $\sigma$-torsion if and only if $\Hom(L_\sigma(H_a),M)=0$ for every $a \in \mathcal S$.
		
		\item The module $P=\bigoplus_{a \in \mathcal S}L_\sigma(H_a)$ is pseudoprojective and $\mathcal T_\sigma=\{X \in \Modl{\mathcal S}\mid \Hom(P,X)=0\}$.
	\end{enumerate}
	When all these conditions are satisfied, the class $\mathcal C_\sigma$ satisfying that $(\mathcal C_\sigma,\mathcal T_\sigma)$ is a torsion theory is precisely $\Gen(P)$ and coincides with the class of all $\sigma$-cotorsion free modules.
\end{proposition}

\begin{proof}
	(1) $\Rightarrow$ (2). Trivial, since $L_\sigma(M)$ is the kernel of the morphism from $M$ to $\prod_{U\in \mathcal U_\sigma(M)}\frac{M}{U}$ induced by all projections, and this latter module is $\sigma$-torsion by (1).
	
	(2) $\Rightarrow$ (3). Trivial.
	
	(3) $\Rightarrow$ (4). If $M$ is $\sigma$-torsion, $\Hom(L_\sigma(H_a),M)=0$ since $L_\sigma(H_a)$ is $\sigma$-cotorsion-free by Lemma \ref{l:L=C}.
	
	Conversely, suppose that $M$ is not $\sigma$-torsion. By Proposition \ref{p:TopologiesAndTorsion}, there exists $a \in \mathcal S$ and $x \in M(a)$ such that $(x:0) \notin \mathcal U_\sigma(H_a)$. This means that $L_\sigma(H_a)$ is not contained in $(x:0)$. Then, the restriction of the canonical morphism $m^x:H_a \rightarrow M$ to $L_\sigma(H_a)$ is non-zero and $\Hom(L_\sigma(H_a),M)\neq 0$.
	
	(4) $\Rightarrow$ (1). If $\{M_i \mid i \in I\}$ is a family of $\sigma$-torsion modules, then $\Hom\left(L_\sigma(H_a),\prod_{i \in I}M_i\right)\cong \prod_{i \in I}\Hom(L_\sigma(H_a),M_i)=0$ for every $a \in \mathcal S$, so that $\prod_{i \in I}M_i$ is $\sigma$-torsion. Then, (1) holds.
	
	(1) $\Rightarrow$ (5). Clearly $\mathcal T_\sigma=\{X \in \Modl{\mathcal S}\mid \Hom_{\mathcal S}(P,X)=0\}$ by (4). Moreover, if $H=\bigoplus_{a \in \mathcal S}H_a$, notice that $L_\sigma(H) \leq P$, since $P$ is $\sigma$-dense in $H$ by (3)~. But $P=\bigoplus_{a \in \mathcal S}C_\sigma(H_a)$ by Lemma \ref{l:L=C}, so that $P \leq C_\sigma(H) \leq L_\sigma(H)$. The conclusion is that $P=L_\sigma(H)$.
	
	Now, using that $\tr_P(H_a)=L_\sigma(H_a)$, we get that $\sum_{a \in \mathcal S}\tr_P(H_a)\cdot P(a)$ is $\sigma$-dense in $P$ by Lemma \ref{l:ProductIsDense}, so that $P\leq \sum_{a \in \mathcal S}\tr_P(H_a)\cdot P(a)$. By Theorem \ref{p:Pseudoprojective}, $P$ is pseudoprojective.
	
	(5) $\Rightarrow$ (1). Trivial.
	
	Finally, that $\Gen(P)$ coincides with $\mathcal C_\sigma$ follows from Proposition \ref{p:Pseudoprojective}. Since $P$ is $\sigma$-cotorsion-free, every module in $\Gen(P)$ is $\sigma$-cotorsion-free as well by Proposition \ref{p:PropertiesCotorsionFree}. Moreover, if $C$ is a $\sigma$-cotorsion-free module, then $\Hom(C,T)=0$ for every $T \in \mathcal T_\sigma$. Then $C \in \mathcal C_\sigma$.
\end{proof}

The closure under pure subobjects of the class of flats of a finitely accessible additive category with enough flats (Corollary \ref{c:EnoughFlatsImpliesPureSubmodules}) will follow from the following result:

\begin{corollary}\label{c:PureSubobjectsJansian}
	If $\sigma$ is a jansian torsion theory, then the class of all $\sigma$-cotorsion-free modules is closed under pure subobjects.
\end{corollary}

\begin{proof}
	By the previous result there exists a pseudoprojective module $P$ such that the class of all $\sigma$-cotorsion-free modules is precisely $\Gen(P)$. Take $M$ a $\sigma$-cotorsion-free module and $K \leq M$ a pure submodule. We prove that $\tr_P(K)=K$.
	
	Take $a \in \mathcal S$ and $x \in K(a)$. Since $M=\sum_{b \in \mathcal S}\tr_P(H_b)\cdot M(b)$ by Proposition \ref{p:Pseudoprojective}, we can find $B \subseteq \mathcal S$ finite and, for every $b \in B$, elements $r_1^b, \ldots, r_{n_b}^b \in \tr_P(H_b)(a)$ and $x_1^b, \ldots, x_{n_b}^b \in M(b)$ such that
	\begin{displaymath}
		x=\sum_{b \in B}\sum_{i=1}^{n_b}M(r_i^b)(x_i^b)
	\end{displaymath}
	Now, consider the following commutative diagram
	\begin{displaymath}
		\begin{tikzcd}
			H_a \arrow{d}{m^x} \arrow{r}{f} & \bigoplus_{b \in B}\bigoplus_{i \in I} F_{b,i} \arrow{d}{g}\\
			K \arrow[hook]{r} & M
		\end{tikzcd}
	\end{displaymath}
	where:
	\begin{itemize}
		\item $f$ is the morphism induced in the product by the family $\left\{f_{b,i} \mid b \in \mathcal S, i \in \{1, \ldots, n_b\}\right\}$. Here, $F_{b,i}=H_b$ and $f_{b,i}=m^{r_i^b}$ for every $b \in B$ and $i \in \{1, \ldots, n_b\}$.
		
		\item $g$ is the morphism induced in the direct sum by the family $\{g_{b,i} \mid b \in \mathcal S, i \in \{1, \ldots, n_b\}\}$. Here, $g_{b,i}=m^{x_i^b}$ for every $b \in B$ and $i \in \{1, \ldots, n_b\}$.
	\end{itemize}
	Using that $K$ is a pure submodule of $M$ we can find $h:\bigoplus_{b \in B}\bigoplus_{i \in I} F_{b,i} \rightarrow K$ such that $hf=m^x$. In particular (we omit the subscript $a$ in the natural transformations):
	\begin{eqnarray*}
		x = m^x(1_a)=hf(1_a)=h\left((r_i^b)_{b \in B, i \leq n_b}\right)=h\left(\sum_{b \in B, i \leq n_b}k_{b,i}(r_i^b)\right)\\
		=h\left(\sum_{b \in B, i \leq n_b}k_{b,i}(r_i^b\circ 1_b)\right)=\sum_{b \in B, i \leq n_b}M(r_i^b)(hk_{b,i}(1_b))
	\end{eqnarray*}
	where $k_{b,i}$ is the inclusion of $F_{b,i}$ into the direct sum. This means that $x \in \sum_{b \in \mathcal S}\left(\tr_P(H_b)\cdot K(b)\right)(a)$, which is contained in $\tr_P(K)$ by Lemma \ref{l:GenMClosedExtensions}. This means that $K=\tr_P(K)$.
\end{proof}

We finish this section with some facts about the torsion pair, $\tau = (\mathcal T_\tau,\mathcal F_\tau)$, cogenerated in $\Modl{\mathcal S}$ by the class of all flat modules. In other words,
\begin{displaymath}
\mathcal T_\tau = \{T \in \Modl{\mathcal S} \mid \Hom(T,X)=0 \quad \forall X \in \Flatl{\mathcal S}\}
\end{displaymath}
and
\begin{displaymath}
\mathcal F_\tau = \{F \in \Modl{\mathcal S} \mid \Hom(Y,F)=0 \quad \forall Y \in \mathcal T_\tau\}.
\end{displaymath}
Notice that $\tau(M) \leq \rej_{\Flatl{\mathcal S}}(M)$ for every module $M$. 

Now we establish some basic properties of the torsion theory $\tau$:

\begin{proposition}\label{p:PropertiesOfTau}
\begin{enumerate}
\item Every $\tau$-torsion free module is isomorphic to a submodule of a direct product of injective envelopes of flat modules.

\item Let $\kappa$ be an infinite regular cardinal and $M$ a module which is $<\kappa$-presented. Then $M$ is $\tau$-torsion if and only if $\Hom(M,T)=0$ for every $<\kappa$-presented flat module $T$.

\item If $M$ is finitely presented, then $M$ is $\tau$-torsion if and only if $\Hom(M,T)=0$ for every finitely generated projective module $T$.

\item The torsion theory $\tau$ is hereditary if and only if the injective envelope of every flat module is $\tau$-torsion free.
		
\item If $M$ is a module with $|M|>w(\mathcal S)$, then every $\tau$-dense submodule of $M$ has the same cardinality as $M$.
		
\item If $\tau$ is hereditary, then any $\tau$-dense submodule of a flat module is essential in $\Modl{\mathcal S}$. In particular, $L_\tau(M)$ contains the socle of $M$ for each flat module $M$.
\end{enumerate}
\end{proposition}

\begin{proof}
(1) Given a $\tau$-torsion free module $M$ and $x \in M$ non-zero, there exists by (1) a non-zero morphism $f:\langle x \rangle \rightarrow F_x$ with $F_x$ flat. This morphism extends to a morphism $g_x:M \rightarrow E(F_x)$, where $E(F_x)$ is the injective envelope of $F_x$. The induced morphism $f:M \rightarrow \prod_{x \in M}E(F_x)$ is then a monomorphism.

(2) Since $\Flatl{\mathcal S}$ is finitely accessible, by \cite[Example 2.13(1)]{AdamekRosicky} it is $\kappa$-accessible for every regular uncountable cardinal $\kappa$. Then, any flat functor is the direct limit of a $\kappa$-direct system of $<\kappa$-presented flat functors. Then the result follows since $\Hom_{\mathcal S}(M,-)$ commutes with $\kappa$-directed limits.

(3) Follows from (2).

(4). If $\tau$ is hereditary, then (b) follows from \cite[Proposition VI.3.2]{Stenstrom}.

Conversely, given $M$ a $\tau$-torsion free module, there exists by (1) a family of flat modules, $\{F_i \mid i \in I\}$, and a monomorphism $f:M \rightarrow \prod_{i \in I}E(F_i)$. This monomorphism factors through the inclusion $e:M \rightarrow E(M)$, that is, there exists $g:E(M) \rightarrow \prod_{i \in I}E(F_i)$ with $ge=f$. Since $e$ is an essential monomorphism and $f$ is monic, $g$ is monic as well. But this means that $E(M)$ is $\tau$-torsion-free, as $\prod_{i \in I}E(F_i)$ is so by hypothesis. Then, the result follows from \cite[Proposition VI.3.2]{Stenstrom}.

(5) Suppose that $K \leq M$ has smaller cardinality than $M$, then we can apply Theorem \ref{t:Purification} to find a pure submodule $N$ of $M$ containing $K$ with the same cardinality than $K$. Then $N \lneq M$, $M/N$ is flat and the canonical projection $M/K \rightarrow M/N$ is nonzero. Then $K$ is not $\tau$-dense.
	
(6) The proof of \cite[5.8]{Golan} holds over functor categories as well.
\end{proof}

\section{When is $\Flatl{\mathcal S}$ abelian?}

In this section we characterize when $\Flatl{\mathcal S}$ is abelian, for which we start studying when it is preabelian. We describe monomorphisms, epimorphisms, kernels and cokernels in $\Flatl{\mathcal S}$. We will see that these notions are related to the torsion theory $\tau$ in $\Modl{\mathcal{S}}$ cogenerated by the flat modules. Given a morphism $f:A \rightarrow B$ in $\Modl{\mathcal S}$, $\Ker f$, $\Coker f$ and $\Img f$ denote its kernel, cokernel and image in the category $\Modl{\mathcal S}$, respectively.

\begin{proposition}\label{p:CharacterizationEpimorphisms}
Let $f:F \rightarrow G$ be a morphism in $\Flatl{\mathcal S}$. Then:
\begin{enumerate}
\item $f$ is an epimorphism in $\Flatl{\mathcal S}$ if and only if $\Img f$ is $\tau$-dense in $G$.

\item $f$ is a monomorphism in $\Flatl{\mathcal S}$ if and only if $f$ is a monomorphism in $\Modl{\mathcal S}$.
\end{enumerate}
\end{proposition}

\begin{proof}
(1) Since the canonical projection $p:G \rightarrow \frac{G}{\Img f}$ is a cokernel of $f$ in $\Modl{\mathcal S}$, there is a bijective correspondence $\Gamma$ between the classes $\{g:G \rightarrow H \mid H \in \Flatl{\mathcal S} \textrm{ and }gf=0\}$ and $\left\{h:\frac{G}{\Img f}\rightarrow H \mid H \in \Flatl{\mathcal S}\right\}$ in such a way that $\Gamma(g)=0$ if and only if $g=0$. Consequently, $f$ is epic in $\Flatl{\mathcal S}$ if and only if the first class consists only of the zero map, if and only if, the same happens with the second class. But this is equivalent to $\Img f$ being $\tau$-dense.

(2) If $f$ is a monomorphism in $\Modl{\mathcal S}$, then it is clearly a monomorphism in $\Flatl{\mathcal S}$. Conversely, if $f$ is a monomorphism in $\Flatl{\mathcal S}$, take a projective presentation of $\Ker f$ in $\Modl{\mathcal S}$, $g:P \rightarrow \Ker f$, let $i:\Ker f \rightarrow F$ be the inclusion and note that $fig=0$. Since $ig$ is a morphism in $\Flatl{\mathcal S}$ and $f$ is monic (in $\Flatl{\mathcal S}$), $ig=0$, which implies that $g=0$. That is, $\Ker f=0$ and $f$ is a monomorphism in $\Modl{\mathcal S}$.
\end{proof}

Now we compute kernels and cokernels in $\Flatl{\mathcal S}$. Recall that flat covers exist in any module category over a ring \cite{BicanBashirEnochs}. The same happens in the functor category $\Modl{\mathcal S}$ by \cite[Corollary 2]{Rump} or \cite[Corollary 3.3]{CriveiPrestTorrecillas}.

\begin{proposition}\label{p:Kernels}
Let $f:F \rightarrow G$ be a morphism in $\Flatl{\mathcal S}$ and denote by $i$ the inclusion $\Ker f \hookrightarrow F$.
\begin{enumerate}
\item If $q:C \rightarrow  \Ker f$ is a flat precover of $\Ker f$, then $iq$ is a pseudokernel of $f$.

\item The morphism $f$ has a kernel in $\Flatl{\mathcal S}$ if and only if the kernel of $f$ in $\Modl{\mathcal S}$ belongs to $\Flatl{\mathcal S}$. In this case, a kernel of $f$ in $\Flatl{\mathcal S}$ is the inclusion $i:\Ker f \rightarrow F$.
\end{enumerate}
\end{proposition}

\begin{proof}
(1) Given $g:F' \rightarrow F$ in $\Flatl{\mathcal S}$ with $fg=0$, there exists $g':F' \rightarrow \Ker f$ with $ig'=g$. Since $q$ is a flat precover, there exists $h:F' \rightarrow C$ with $qh=g'$. Then $iqh=g$ which implies, since $g$ was arbitrary, that $iq$ is a pseudocokernel of $f$.

(2) Suppose that $k:K \rightarrow F$ is a kernel of $f$ in $\Flatl{\mathcal S}$. Since $\Ker f$ is the kernel of $f$ in $\Modl{\mathcal S}$, there exists a unique $k':K \rightarrow \Ker f$ such that $ik'=k$. Now take $q:C \rightarrow \Ker f$ the flat cover of $\Ker f$ in $\Modl{\mathcal S}$. Since $f
iq=0$, there exists $w:C \rightarrow K$ with $kw=iq$. Moreover, since $K$ is flat, there exists $v:K \rightarrow C$ such that $qv=k'$. 

Now notice that $iqvw=iq$ and, since $i$ is a monomorphism in $\Modl{\mathcal S}$, we get that $qvw=q$. Using that $q$ is a flat cover, we conclude that $vw$ is an isomorphism in $\Modl{\mathcal S}$ which implies that $v$ is an epimorphism in $\Modl{\mathcal S}$. Then $k'=qv$ is an epimorphisms in $\Modl{\mathcal S}$, since the flat precovers are epic.

Finally, $k$, being a kernel, is a monomorphism in $\Flatl{\mathcal S}$. By Proposition \ref{p:CharacterizationEpimorphisms}, $k$ is a monomorphism in $\Modl{\mathcal S}$, and so is $k'$, since $k=ik'$. The conclusion is that $k'$ is an isomorphism and that $\Ker f \in \Flatl{\mathcal S}$.

Conversely, if the kernel of $f$ (in $\Modl{\mathcal S}$) belongs to $\Flatl{\mathcal S}$, then it is the kernel of $f$ in $\Flatl{\mathcal S}$.
\end{proof}

This result has two immediate consequences. First, a criterion to check if a morphism is a kernel in $\Flatl{\mathcal S}$.

\begin{corollary}
Let $k:K \rightarrow F$ be a morphism in $\Modl{\mathcal S}$. Then $k$ is a kernel in $\Flatl{\mathcal S}$ if and only if $K$ and $F$ are flat, $k$ is monic in $\Modl {\mathcal S}$ and $\frac{F}{\Img k}$ embeds in a flat module.
\end{corollary}

\begin{proof}
Suppose that $k$ is the kernel of a morphism $f:F \rightarrow G$ in $\Flatl{\mathcal S}$. Then $k$ is monic in $\Modl{\mathcal S}$ by Proposition \ref{p:CharacterizationEpimorphisms}, and $\Img k = \Ker f$ by Proposition \ref{p:Kernels}. Then $\frac{F}{\Img k} \cong \Img f$ embeds in a flat module.

Conversely, suppose that $K$ and $F$ are flat, $k$ is monic in $\Modl{\mathcal S}$ and there exists a monomorphism $u:\frac{F}{\Img k} \rightarrow G$ in $\Modl{\mathcal S}$ with $G$ a flat module. Then, a kernel in $\Modl{\mathcal S}$ of $up$, where $p$ is the canonical projection $F \rightarrow \frac{F}{\Img k}$, is $k$. But $up$ is a morphism in $\Flatl{\mathcal S}$, so that $k$ is actually a kernel in $\Flatl{\mathcal S}$.
\end{proof}

The second consequence is the characterization of when $\Flatl{\mathcal S}$ has kernels.

\begin{corollary}\label{c:ExistenceKernels}
The following assertions are equivalent for the small preadditive category $\mathcal S$:
\begin{enumerate}
\item $\Flatl{\mathcal S}$ has kernels.

\item $\Flatl{\mathcal S}$ is closed under kernels.

\item $\Modl{\mathcal S}$ has weak dimension less than or equal to $2$.
\end{enumerate}
\end{corollary}

\begin{proof}
(1) $\Leftrightarrow$ (2). Follows from Proposition \ref{p:Kernels}.

(2) $\Rightarrow$ (3). For any module $M$, any flat resolution,
\begin{displaymath}
\begin{tikzcd}
\cdots \arrow{r} & F_1 \arrow{r}{f_1} & F_0 \arrow{r}{f_0} & M \arrow{r} & 0,
\end{tikzcd}
\end{displaymath}
satisfies that $\Ker f_1$ is flat by (2). This means that the flat dimension of $M$ is less than or equal to $2$.

(3) $\Rightarrow$ (2). Given any morphism $f$ in $\Flatl{\mathcal S}$, $\Coker f$ has flat dimension less than or equal to $2$. This implies that $\Ker f$ is flat, so that (2) holds.
\end{proof}

The situation with cokernels is different. First, there might not be flat preenvelopes in $\Modl{\mathcal S}$ unless $\Modr{\mathcal S}$ is locally coherent by Corollary \ref{c:CoherentCategories}. Second, flat preenvelopes might not be monomorphisms. As we will see, cokernels are related with flat reflections. We start characterizing flat reflections in terms of $\tau$-dense submodules. We call a submodule $K$ of a flat module $F$ \textit{strongly $\tau$-dense} if it is $\tau$-dense and the inclusion $K \hookrightarrow F$ is a flat preenvelope.

\begin{lemma}\label{l:Reflection}
Let $f:M \rightarrow F$ be a morphism with $F$ flat. Then the following assertions are equivalent:
\begin{enumerate}
	\item $f$ is a flat reflection.
	
	\item $f$ is a flat preenvelope and $\Img f$ is $\tau$-dense in $F$.
	
	\item $\Hom(\overline f,F')$ is an epimorphism for every flat module $F'$ (where $\overline f:M \rightarrow \Img f$ is the canonical morphism) and $\Img f$ is strongly $\tau$-dense in $F$.
\end{enumerate}
\end{lemma}

\begin{proof}
	(1) $\Leftrightarrow$ (2). It is clear since $f$ is a flat reflection if and only if $f$ is a flat preenvelope and $\Hom(f,F')$ is monic for every flat module $F'$. But this is equivalent to $F/\Img f$ being $\tau$-torsion.
	
	(2) $\Leftrightarrow$ (3). If $i:\Img f \rightarrow F$ is the inclusion, then 
	\begin{displaymath}
		\Hom(f,F')=\Hom(\overline f,F') \circ \Hom(i,F')
	\end{displaymath}
	from which it is very easy to see that $f$ is a flat preenvelope if and only if $i$ is a flat preenvelope and $\Hom(\overline f,F')$ is an epimorphism. From this fact it follows the equivalence (2) $\Leftrightarrow$ (3).
\end{proof}

Now we study pseudocokernels and cokernels in $\Flatl{\mathcal S}$.

\begin{lemma}\label{l:Cokernels}
Let $f:F \rightarrow G$ be a morphism in $\Flatl{\mathcal S}$ and denote by $p:G \rightarrow \frac{G}{\Img f}$ the canonical projection.
\begin{enumerate}
\item If $e:\frac{G}{\Img f} \rightarrow H$ is a flat preenvelope of $\frac{G}{\Img f}$, then $ep$ is a pseudocokernel of $f$.

\item If $c:G \rightarrow H$ is a cokernel of $f$ in $\Flatl{\mathcal S}$, then the unique morphism $h:\frac{G}{\Img f} \rightarrow H$ such that $hp=c$ is a flat reflection of $\frac{G}{\Img f}$.

\item If $h:\frac{G}{\Img f} \rightarrow H$ is a flat reflection of $\frac{G}{\Img f}$, then $hp$ is a cokernel of $f$ in $\Flatl{\mathcal S}$.
\end{enumerate}
\end{lemma}

\begin{proof}
(1) The proof is dual to the one of (1) of Proposition \ref{p:Kernels}.

(2) Since $c$ is an epimorphism in $\Flatl{\mathcal S}$ and $hp=c$, $h$ is an epimorphism in $\Flatl{\mathcal S}$. By Proposition \ref{p:CharacterizationEpimorphisms}, $\Img h$ is $\tau$-dense in $H$. 

Now we see that $h$ is a flat preenvelope. Given any $g:\frac{G}{\Img f} \rightarrow M$ with $M \in \Flatl{\mathcal S}$, since $gpf=0$, we can find $h':H \rightarrow M$ with $h'c=gp$. Then $h'hp=gp$ and, as $p$ is an epimorphism in $\Modl{\mathcal{S}}$, $h'h=g$. Then $h$ is a flat reflection by Lemma \ref{l:Reflection}.

(3) Given $g:G \rightarrow M$ a morphism in $\Flatl{\mathcal S}$ with $gf=0$, we can find a unique $g_1:\frac{G}{\Img f}\rightarrow M$ satisfying $g_1p=g$, since $p$ is a cokernel of $f$ in $\Modl{\mathcal S}$, and a unique $g_2:H \rightarrow M$ satisfying $g_2h=g_1$, since $h$ is a flat reflection. It is easy to see that $g_2$ is unique satisfying $g_2hp=g$, so that $hp$ is a cokernel of $f$ in $\Flatl{\mathcal S}$.
\end{proof}

Again, we obtain two consequences of this result. First, a criterion to determine when a morphism is a cokernel in $\Flatl{\mathcal S}$.

\begin{proposition}\label{p:CharacterizationCokernels}
Let $c:F \rightarrow C$ be a morphism in $\Modl{\mathcal S}$. Then $c$ is a cokernel in $\Flatl{\mathcal S}$ if and only if $F$ and $C$ are flat and $\Img c$ is a strongly $\tau$-dense submodule of $C$.
\end{proposition}

\begin{proof}
If $c$ is the cokernel in $\Flatl{\mathcal S}$ of a morphism $f:G \rightarrow F$ and we denote by $p:F \rightarrow \frac{F}{\Img f}$ the canonical morphism, then the unique morphisms $h:\frac{F}{\Img f} \rightarrow C$ satisfying $hp=c$ is a flat reflection by the previous lemma. By Lemma \ref{l:Reflection}, $\Img c = \Img h$ is strongly $\tau$-dense in $C$.

Conversely, take an epimorphism in $\Modl{\mathcal S}$, $f:G \rightarrow \Ker c$ with $G$ a flat module. Let $\overline c:\frac{F}{\Ker c} \rightarrow \Img c$ be the canonical isomorphism, and $k:\Ker c \rightarrow F$ and $i:\Img c \rightarrow C$ the inclusions.. We prove that $i\overline c$ is a flat reflection, since this implies that $i\overline c p = c$ is a cokernel of $kf$ in $\Modl{\mathcal S}$ by Lemma \ref{l:Cokernels}. But this follows directly from Lemma \ref{l:Reflection}(3).
\end{proof}

As a second consequence, we characterize when $\Modl{\mathcal S}$ has cokernels using a recent characterization of reflective subcategories given in \cite{CortesCriveiSaorin}. The remarkable fact here is that if $\Flatl{\mathcal S}$ has cokernels, then it is automatically preabelian, i.e., it has kernels.

\begin{theorem}\label{t:Preabelian}
The following are equivalent for the small preadditive category $\mathcal S$.
\begin{enumerate}
\item $\Flatl{\mathcal S}$ has cokernels.

\item $\Flatl{\mathcal S}$ is a reflective subcategory of $\Modl{\mathcal S}$.

\item $\Flatl{\mathcal S}$ is preenveloping and closed under kernels.

\item $\Modr{\mathcal S}$ is locally coherent and $\Modl{\mathcal S}$ has weak dimension less than or equal to $2$.

\item $\Flatl{\mathcal S}$ is preabelian.
\end{enumerate}
\end{theorem}

\begin{proof}
(1) $\Rightarrow$ (2). For any module $M$ we can find an exact sequence in $\Modl{\mathcal S}$,
\begin{displaymath}
\begin{tikzcd}
F_1 \arrow{r}{f_1} & F_0 \arrow{r}{f_0} & M \arrow{r} & 0,
\end{tikzcd}
\end{displaymath}
with $F_0, F_1 \in \Flatl{\mathcal S}$; since $f_1$ has a cokernel in $\Flatl{\mathcal S}$, and $f_0$ is a cokernel of $f_1$ in $\Modl{\mathcal S}$, $M$ has a flat reflection by Lemma \ref{l:Cokernels}. This means that $\Flatl{\mathcal S}$ is a reflective subcategory of $\Modl{\mathcal S}$. 

(2) $\Rightarrow$ (1). Follows from Lemma \ref{l:Cokernels}.

(2) $\Leftrightarrow$ (3). By \cite[Corollary 7.2]{CortesCriveiSaorin}.

(3) $\Leftrightarrow$ (4). By Corollaries \ref{c:CoherentCategories} and \ref{c:ExistenceKernels}.

(1) and (3) $\Rightarrow$ (5) $\Rightarrow$ (1) are trivial.
\end{proof}

Notice that, as we mentioned before, $\Flatl{\mathcal S}$ might have cokernels but it might not be closed under cokernels. Actually, since every module is the cokernel in $\Modl{\mathcal S}$ of a morphism in $\Flatl{\mathcal S}$, $\Flatl{\mathcal S}$ is closed under cokernels if and only if $\Flatl{\mathcal S}=\Modl{\mathcal S}$.

Now we characterize when $\Flatl{\mathcal S}$ is abelian. Recall that a preabelian category $\mathcal A$ is \textit{abelian} if for any morphism $f$, the canonical morphism $\overline f:\Coim f \rightarrow \Img f$ (sometimes called the \textit{parallel} of $f$ \cite[p. 24]{Popescu}) is an isomorphism, where $\Coim f$ is the cokernel of the kernel of $f$ and $\Img f$ is the kernel of the cokernel of $f$. We will use the following observation:

\begin{lemma}\label{l:RejFlatDenseSubmodules}
	Let $K$ be a submodule of a flat module $G$ and suppose that $K \hookrightarrow \rej_{\Flatl{\mathcal S}}^K(G)$ is a flat reflection. If $K$ is flat, then $\rej_{\Flatl{\mathcal S}}^K(G)=K$.
\end{lemma}

\begin{proof}
	If $K$ is flat, then $K=\rej_{\Flatl{\mathcal S}}^K(G)$ by the minimality of the flat reflection of a module. If $K$ is $\tau$-dense, then $\rej_{\Flatl{\mathcal S}}(G/K)=G/K$, so that $\rej_{\Flatl{\mathcal S}}^K(G)=G$.
\end{proof}

\begin{theorem}\label{t:Panoramic}
The category $\Flatl{\mathcal S}$ is abelian if and only if  $\Modr{\mathcal S}$ is locally coherent and for any submodule $K$ of a flat module $G$, the inclusion $K \hookrightarrow \rej_{\Flatl{\mathcal S}}^K(G)$ is a flat reflection of $K$.
\end{theorem}

\begin{proof}
By Theorem \ref{t:Preabelian}, $\Flatl{\mathcal S}$ is preabelian if and only if $\Modr{\mathcal S}$ is locally coherent and has weak dimension less than or equal to $2$. Moreover, if $K \hookrightarrow \rej_{\Flatl{\mathcal S}}^K(G)$ is a flat reflection for every submodule $K$ of a flat module $G$, then $\Flatl{\mathcal S}$ is preabelian, since if $f:F \rightarrow G$ is a morphism in $\Flatl{\mathcal S}$, then $\Ker f = \rej_{\Flatl{\mathcal S}}^{\Ker f}(F)$ and it has to be flat. Consequently, in order to prove that (1) and (2) are equivalent, we may assume that $\Flatl{\mathcal S}$ is preabelian.

For every morphism $f:F \rightarrow G$ in $\Flatl{\mathcal S}$ take, using Theorem \ref{t:Preabelian}, flat reflections $r:F/\Ker f \rightarrow H$ and $r':G/\Img f \rightarrow H'$. We can construct the following commutative diagram
\begin{equation}\label{e:Diagrama}
\begin{tikzcd}
\Ker f \arrow{r}{k} & F \arrow{r}{f} \arrow{d}{p} & G \arrow{r}{p'} & \frac{G}{\Img f} \arrow{r}{r'} & H'\\
 & \frac{F}{\Ker f} \arrow{d}{r} \arrow{r}{\overline f} & \Img f \arrow{u}{i} \arrow{d}{i'} \\
 & H \arrow[dotted]{r}{g}& \rej_{\Flatl{\mathcal S}}^{\Img f}(G)
\end{tikzcd}
\end{equation}
where $k$, $i$ and $i'$ are inclusions, $p$ and $p'$ projections and $\overline f$ is the parallel of $f$ in $\Modl{\mathcal S}$. The morphism $g$ is constructed as follows: By Lemma \ref{l:Cokernels}, $r'p'$ is de cokernel of $f$ in $\Flatl{\mathcal S}$. Clearly, the kernel of $r'p'$ in $\Modl{\mathcal S}$ is $\rej_{\Flatl{\mathcal S}}^{\Img f}(G)$, so that it is the kernel of $r'p'$ in $\Flatl{\mathcal S}$ by Proposition \ref{p:Kernels} or, in other words, the image of $f$ in $\Flatl{\mathcal S}$. Denote by $j$ the inclusion $\rej_{\Flatl{\mathcal S}}^{\Img f}(G) \hookrightarrow G$. Similarly, $rp$ is the cokernel of $k$ in $\Flatl{\mathcal S}$, which means that $H$ is the coimage of $f$ in $\Flatl{\mathcal S}$. Let $g:H \rightarrow \rej_{\Flatl{\mathcal S}}^{\Img f}(G)$ be the parallel of $f$ in $\Flatl{\mathcal S}$. Then $f=jgrp=i\overline f p$ and, using that $i=ji'$ we conclude that $gr=i'\overline f$.

Now, using diagram (\ref{e:Diagrama}) we can prove, assuming that $\Flatl{\mathcal S}$ is preabelian, that $\Flatl{\mathcal S}$ is abelian if and only if for every submodule $K$ of a flat module $G$, the inclusion $K \hookrightarrow \rej_{\Flatl{\mathcal S}}^K(G)$ is a flat reflection of $K$. If $\Flatl{\mathcal S}$ is abelian and $K$ is a submodule of a flat module $G$, we can find a morphism $f:F \rightarrow G$ with $F$ flat and $\Img f=K$. Then, in diagram (\ref{e:Diagrama}) $\overline f$ and $g$ are isomorphisms, so that $i'=gr\overline f^{-1}$ is a flat reflection.

Conversely, if in diagram (\ref{e:Diagrama}) $i'$ is a flat reflection of $\Img f$, then $\Img i'\leq \Img g$, so that $\Img g$ is $\tau$-dense in $\rej_{\Flatl{\mathcal S}}^{\Img f}(G)$ by Lemma \ref{l:Reflection}. Moreover, using that $r$ and $i'$ are flat reflections and $\overline f$ is an isomorphism, it is easy to see that $g$ is a flat preenvelope. By Lemma \ref{l:Reflection}, $g$ is a flat reflection and, since $H$ is flat, it has to be an isomorphism.
\end{proof}

Our characterization of when $\Flatl{\mathcal S}$ is abelian seems to be new even in the case of rings (compare with \cite[Theorem 2.7]{GarciaSimson}). Now, we give some consequences of our result:

\begin{corollary}\label{c:TorsionPanoramic}
Suppose that $\Flatl{\mathcal S}$ is abelian. Then: 
\begin{enumerate}
\item The global weak dimension of $\Modl{\mathcal S}$ is $0$ or $2$.
	
\item A flat module $E$ is injective in $\Flatl{\mathcal{S}}$ if and only if it is injective in $\Modl{\mathcal S}$.

\item The torsion theory $\tau$ is hereditary, $\mathcal F_\tau$ consists of all modules isomorphic to a submodule of a flat module and $\tau=\rej_{\Flatl{\mathcal S}}$.

\item For a nonzero module $M$, $M$ is $\tau$-torsion free if and only if its flat dimension is less than or equal to $1$.

\item A module $M$ decomposes as $T \oplus P$ with $T$ being $\tau$-torsion and $P$, projective, if and only if $\Ext^1(M,F)=0$ for every flat module $F$.
\end{enumerate}
\end{corollary}

\begin{proof}
(1) We know, from Theorem \ref{t:Preabelian}, that the global weak dimension is less than or equal to $2$. So, let us assume that the weak dimension is greater than or equal to $1$ and let us see that it is equal to $2$.

First, we see that $\mathcal T_\tau=0$. Take $M \in \mathcal T_\tau$ and a short exact sequence
\begin{displaymath}
	0 \rightarrow F \hookrightarrow G \xrightarrow{h} M \rightarrow 0
\end{displaymath}
with $F$ and $G$ flat modules. By Lemma \ref{l:RejFlatDenseSubmodules}, $F=\rej_{\Flatl{\mathcal S}}^F(G)$, so that $\rej_{\Flatl{\mathcal S}}(M)=0$. This means that $M$ is $\tau$-torsion free and, consequently, $M=0$.

Now, take any module $M$, a short exact sequence  
\begin{displaymath}
	0 \rightarrow F \hookrightarrow G \xrightarrow{h} M \rightarrow 0
\end{displaymath}
with $F$ and $G$ flat modules and $r:M \rightarrow H$ a flat reflection, which exists by Theorem \ref{t:Preabelian}. By Lemma \ref{l:RejFlatDenseSubmodules} again, $F=\rej_{\Flatl{\mathcal S}}^F(G)$, so that $\Ker r=\rej_{\Flatl{\mathcal S}}(M)=0$. But $\Img r$ is $\tau$-dense in $H$ by Lemma \ref{l:Reflection}, and, using that $\mathcal T_\tau=0$, we conclude that $\Img r = H$. This implies that $M \cong H$ is flat. Since $M$ is arbitrary, the conclusion is that the weak dimension of $\mathcal S$ is $0$.	

(2) If $E$ is injective in $\Flatl{\mathcal S}$ and we take a morphism $f:I \rightarrow E$ defined from a submodule $I$ of a representable functor $P$, since the inclusion $I \hookrightarrow \rej_{\Flatl{\mathcal S}}^I(P)$ is a flat reflection by Theorem \ref{t:Panoramic}, we can extend $f$ to a $g:\rej_{\Flatl{\mathcal S}}^I(P) \rightarrow E$; but $\rej_{\Flatl{\mathcal S}}^I(P)$ is flat and $E$ is injective in $\Flatl{\mathcal S}$, so that we can extend $g$ to $h:P \rightarrow E$. Consequently, $E$ is injective by \cite[Proposition V.2.9]{Stenstrom}.

(3) Take $F$ a flat module and denote by $u:F \rightarrow E(F)$ its injective hull in $\Modl{\mathcal S}$. Since $\Flatl{\mathcal S}$ is Grothendieck, there exists an injective hull $v:F \rightarrow E'(F)$ of $F$ in $\Flatl{\mathcal S}$. Now $v$ is a monomorphism in $\Modl{\mathcal S}$ by Proposition \ref{p:CharacterizationEpimorphisms} and $E'(F)$ is injective in $\Modl{\mathcal S}$ by (1), so that $E(F)$ is isomorphic to a direct summand of $E'(F)$. In particular, $E(F)$ is flat and $\tau$ is hereditary by Proposition \ref{p:PropertiesOfTau}. 

The second assertion follows from the fact that every $\tau$-torsion free module is isomorphic to a submodule of a direct product of injective envelopes of flat modules, Proposition \ref{p:PropertiesOfTau}.

Finally, in order to prove that $\tau=\rej_{\Flatl{\mathcal S}}$ notice that we always has the inclusion $\tau(M) \leq \rej_{\Flatl{\mathcal S}}(M)$. Since $M/\tau(M)$ is isomorphic to a submodule of a flat module for every module $M$, we actually have the equality $\tau(M) = \rej_{\Flatl{\mathcal S}}(M)$.

(3) Suppose that $M$ has flat dimension smaller than $2$. Then, there exists an epimorphism in $\Modl{\mathcal S}$, $f:F \rightarrow M$, whose kernel, $K$, is flat. By Theorem \ref{t:Panoramic} and Lemma \ref{l:RejFlatDenseSubmodules}, $\rej_{\Flatl{\mathcal S}}^K(F)=K$ which means that $\rej_{\Flatl{\mathcal S}}(M)=0$. Using Lemma \ref{l:CogFlats} and the fact that product of flats are flat, we conclude $M$ is isomorphic to a submodule of a flat module. That is, $M$ is $\tau$-torsion free.

Conversely, suppose that the flat dimension of $M$ is $2$ and take a flat presentation of $M$, $f:F \rightarrow M$, with kernel $K$. Then, the inclusion $K \hookrightarrow \rej_{\Flatl{\mathcal S}}^K(F)$ is a flat reflection of $K$ by Theorem \ref{t:Panoramic}. Since $K$ is not flat, the quotient $\rej_{\Flatl{\mathcal S}}^K(F)/K$ is a nonzero torsion module, by Lemma \ref{l:Reflection}, which is isomorphic to a submodule of $M$. This means that $M$ is not torsion-free.

(4) First, suppose that $M\cong T \oplus P$. Then, $\Ext^1(M,F)\cong\Ext^1(T,F)$ , so that let us prove that $\Ext^1(T,F)=0$ for every flat module $F$.

Take $K$ a $\tau$-dense submodule of a projective module $P$ such that $P/K \cong T$ and fix a flat module $F$. Applying $\Hom(-,F)$ to the short exact sequence
\begin{displaymath}
	\begin{tikzcd}
		0 \arrow{r} & K \arrow{r}{i} & P \arrow{r} & T \arrow{r} & 0
	\end{tikzcd}
\end{displaymath}
we obtain the long exact sequence:
\begin{displaymath}		
	\begin{tikzcd}
		0 \arrow{r} & \Hom(P,F) \arrow{r} & \Hom(K,F)\arrow{r} & \Ext^1(T,F) \arrow{r} & \Ext^1(P,F)
	\end{tikzcd}
\end{displaymath}
Since $\rej_{\Flatl{\mathcal S}}^K(P)=P$, Theorem \ref{t:Panoramic} says that the inclusion $i:K \rightarrow P$ is a flat reflection. Using that $\Hom(i,F)$ is an isomorphism and that $\Ext^1(P,F)=0$, we conclude that $\Ext^1(T,F)=0$.

Conversely, suppose that $\Ext^1(M,F)=0$ for every flat module $F$ and take a short exact sequence
\begin{displaymath}
	\begin{tikzcd}
		0 \arrow{r} & K \arrow{r}{i} & P \arrow{r} & M \arrow{r} & 0
	\end{tikzcd}
\end{displaymath}
with $P$ projective. The condition $\Ext^1(M,F)$ for every flat module $F$ implies that $i$ is a flat preenvelope. Since, by Theorem \ref{t:Panoramic}, the inclusion $K \hookrightarrow \rej_{\Flatl{\mathcal S}}^K(P)$ is a flat reflection, $P=\rej_{\Flatl{\mathcal S}}^K(P)\oplus P'$ for some submodule $P'$ of $P$. Then
\begin{displaymath}
	M \cong \frac{\rej_{\Flatl{\mathcal S}}^K(P)}{K} \oplus P'
\end{displaymath}
where $P'$ is projective and $\rej_{\Flatl{\mathcal S}}^K(P)/K$ is $\tau$-torsion, since it is equal to $\tau(P/K)$ by (3).
\end{proof}

\section{Flat and projective objects in $\Flatl{\mathcal S}$}

There are two canonical ways of constructing an exact structure in $\Flatl{\mathcal S}$ (there are others, see \cite{Crivei}). First, since $\Flatl{\mathcal S}$ is closed under extensions in $\Modl{\mathcal S}$ (which means that each short exact sequence $0 \rightarrow A \rightarrow B \rightarrow C \rightarrow 0$ in $\Modl{\mathcal S}$ with $A,C \in \Flatl{\mathcal S}$, satisfies that $B \in \Flatl{\mathcal S}$), $\Flatl{\mathcal S}$ is an exact category in which the class $\mathcal P$ of exact sequences consists of the exact sequences in $\Modl{\mathcal S}$ with all terms belonging to $\Flatl{\mathcal S}$.

Second, since $\Flatl{\mathcal S}$ is a finitely accessible additive category, we can consider in $\Flatl{\mathcal S}$ the class $\mathcal P'$ consisting of all of composable morphisms, $A \rightarrow B \rightarrow  C$, such that
\begin{displaymath}
\begin{tikzcd}
0 \arrow{r} & \Hom(P,A) \arrow{r} & \Hom(P,B) \arrow{r} & \Hom(P,C) \arrow{r} & 0
\end{tikzcd}
\end{displaymath}
is exact in the category of abelian groups for every module $P$ which is finitely presented in $\Flatl{\mathcal S}$ (i.e., finitely generated projective in $\Modl{\mathcal S}$). As a consequence of the following result, which is true for any Grothendieck category, both exact structures $\mathcal P$ and $\mathcal P'$ coincide.

\begin{lemma}\label{l:Exactness}
Let $\mathcal A$ be a Grothendieck category with a set $\mathcal G$ of projective generators. A sequence of morphisms
\begin{displaymath}
\begin{tikzcd}
0 \arrow{r} & A \arrow{r}{f} & B \arrow{r}{g} & C \arrow{r} & 0
\end{tikzcd}
\end{displaymath}
is short exact provided that the sequence
\begin{equation}\label{e:Sequence}
\begin{tikzcd}
0 \arrow{r} & \Hom_{\mathcal A}(G,A) \arrow{r} & \Hom_{\mathcal A}(G,B) \arrow{r} & \Hom_{\mathcal A}(G,C) \arrow{r} & 0
\end{tikzcd}
\end{equation}
is short exact in $\mathbf{Ab}$ for every $G \in \mathcal G$.
\end{lemma}

\begin{proof}
In order to see that $g$ is an epimorphism take $w:\bigoplus_{i\in I}G_i \rightarrow C$ an epimorphism with $G_i \in \mathcal G$ for each $i \in I$. Then, using that $\Hom_{\mathcal A}(G_i,g)$ is an epimorphism for each $i \in I$ it is easy to see that $w$ factors through $g$. In particular, $g$ is an epimorphism.

In order to see that $f$ is monic take $t:\bigoplus_{j \in J}H_j \rightarrow \Ker f$ an epimorphism with $H_j \in \mathcal G$ for each $j \in J$. Then, if $k_j: H_j \rightarrow \bigoplus_{j \in J}H_j$ and $k:\Ker f \rightarrow A$ are the canonical monomorphisms, using that $\Hom_{\mathcal A}(H_j,f)(kk_j)=0$ and that $\Hom_{\mathcal A}(H_j,f)$ is monic, we get that $kk_j=0$ for each $j \in J$. This means that $k=0$, that $\Ker f=0$ and that $f$ is monic.

Finally, let us see that the sequence is exact at $B$. In order to see that $gf=0$, take an epimorphism $v:\bigoplus_{l \in L}T_l \rightarrow A$ with $T_l \in \mathcal G$ for every $l \in L$, and notice that, since $\Hom_{\mathcal A}(G,gf)=0$, we get that $gfv=0$. Since $v$ is epic, $gf=0$. Now, we can construct the following commutative diagram:
\begin{displaymath}
\begin{tikzcd}
A \arrow{r}{f} \arrow{d}{p} & B \arrow{r}{g} & C & \\
\Coim f \arrow{r}{\overline f} & \Img f \arrow{u}{j} \arrow{r}{k} & \Ker g \arrow{r}{q} \arrow{lu}{i} & \Coker k
\end{tikzcd}
\end{displaymath}
where $k$ exists since $gj=0$. We only have to see that $k$ is an epimorphism and, consequently, an isomorphism. Let $q:\Ker g \rightarrow \Coker k$ be the cokernel of $k$ and take an epimorphism $u:\bigoplus_{l \in L}T_l \rightarrow \Ker g$ with $T_l \in \mathcal G$; denote by $k_l:T_l \rightarrow \bigoplus_{l \in L}T_l$ the canonical inclusion for every $l \in L$.  Using that the sequence (\ref{e:Sequence}) is exact and that $giuk_l=0$ we can find $h_l\colon T_l \rightarrow A$ with $fh_l=iuk_l$ for each $l \in L$. These morphisms induce $h\colon\bigoplus_{l \in L}T_l \rightarrow A$ which satisfies $fh=iu$. Then $ik\overline fph=iu$ and, since $i$ is monic, $k\overline fph=u$. Then $qu=0$ which implies that $q=0$ and, consequently, $\Coker k=0$.
\end{proof}

We call the exact sequences in $\mathcal P=\mathcal P'$ \textit{pure-exact sequences}.  A morphism $f:M \rightarrow N$ is a \textit{pure monomorphism} (resp. a \textit{pure epimorphism}) if there exists a pure-exact sequence $0\rightarrow K \rightarrow M \xrightarrow{f} N \rightarrow 0$ (resp. $0 \rightarrow M \xrightarrow{f} N \rightarrow Q \rightarrow 0$).

The notion of flat object in a module category or in a finitely accessible abelian category $\mathcal A$ is clear \cite{Crawley}, \cite{Stenstrom68} and \cite{CuadraSimson}: an object $A$ of $\mathcal A$ is flat if any epimorphism $f:B \rightarrow A$ in $\mathcal A$ is pure. However, if the category $\mathcal A$ is not abelian, there have been considered several notions of flatness in $\mathcal A$. Here, we consider the one used in \cite{CriveiPrestTorrecillas} and in \cite{Crivei}. 

\begin{definition}
Given a module $F \in \Flatl{\mathcal S}$, we say that $F$ is
\begin{enumerate}
\item \textit{Projective} if every epimorphism in $\Flatl{\mathcal S}$, $f:H \rightarrow F$, is a split epimorphism.

\item \textit{Flat}, if every epimorphism in $\Flatl{\mathcal S}$, $f:H \rightarrow F$, is pure.
\end{enumerate}
\end{definition}

Notice that this notion of flatness is different from the one considered in \cite{Rump}, where Rump fixes a right exact structure $\mathcal E$ in the finitely accessible additive category $\mathcal A$ containing all pure epimorphisms, and define an object $A$ of $\mathcal A$ to be flat if every conflation onto $A$ is a pure epimorphism. However, both notions of flatness coincide in any finitely accessible Grothendieck category considering the abelian exact structure. 

The following is a characterization of flat objects in $\Flatl{\mathcal S}$ in terms of the torsion theory $\tau$.

\begin{proposition}\label{p:CharacterizationFlats}
Let $F$ be a module in $\Flatl{\mathcal S}$. Then:
\begin{enumerate}
\item $F$ is flat in $\Flatl{\mathcal S}$ if and only if $F$ is $\tau$-cotorsion-free (i. e., it contains no proper dense submodule).

\item $F$ is projective in $\Flatl{\mathcal S}$ if and only if $F$ is flat in $\Flatl{\mathcal S}$ and projective in $\Modl{\mathcal S}$.
\end{enumerate}
\end{proposition}

\begin{proof}
(1) Suppose that $F$ is flat in $\Flatl{\mathcal S}$ and let $K \leq F$ be such that $F/K$ is $\tau$-torsion. Take $f:G \rightarrow F$ a morphism in $\Flatl{\mathcal S}$ with $\Img f= K$. By Proposition \ref{p:CharacterizationEpimorphisms}, $f$ is an epimorphism in $\Flatl{\mathcal S}$. Since $F$ is flat in $\Flatl{\mathcal S}$, $f$ is a pure epimorphism in $\Flatl{\mathcal S}$ and, in particular, an epimorphism in $\Modl{\mathcal S}$ by Lemma \ref{l:Exactness}, so that $K=F$.

The converse is clear by Proposition \ref{p:CharacterizationEpimorphisms}.

(2) Suppose that $F$ is projective in $\Flatl{\mathcal S}$. Then it is clearly flat in $\Flatl{\mathcal S}$ since every split epimorphism is pure. Moreover, there exists an epimorphism in $\Modl R$, $f:R^{(I)} \rightarrow F$. Since $F$ is projective in $\Flatl{\mathcal S}$, this epimorphism splits in $\Flatl{\mathcal S}$ and, equivalently, in $\Modl R$. That is, $F$ is projective in $\Modl R$.

The converse is clear.
\end{proof}

%

The following result collects the main properties of flat objects in $\Flatl{\mathcal S}$.

\begin{proposition}\label{p:PropertiesFlats}
\begin{enumerate}
\item The class of all flat modules in $\Flatl{\mathcal S}$ is closed under pure quotients.

\item If $f:F \rightarrow G$ is an epimorphism in $\Flatl{\mathcal S}$ such that both $\Ker f$ and $G$ are flat in $\Flatl{\mathcal S}$, then $F$ is flat in $\Flatl{\mathcal S}$.

\item If $\tau$ is hereditary, the class of flat modules in $\Flatl{\mathcal S}$ is closed under direct limits.

\item If $\tau$ is jansian, then the class of flat objects in $\Flatl{\mathcal S}$ is closed under pure submodules.
\end{enumerate}
\end{proposition}

\begin{proof}
(1) Given a pure monomorphism $f:F \rightarrow G$ in $\Flatl{\mathcal S}$, $f$ is an epimorphism in $\Modl{\mathcal S}$ and $G$ is $\tau$-cotorsion-free by Proposition \ref{p:PropertiesCotorsionFree}. By Proposition \ref{p:CharacterizationFlats}, $G$ is flat in $\Flatl{\mathcal S}$.

(2) Since $G$ is flat in $\Flatl{\mathcal S}$, $f$ is a pure epimorphism, i. e., an epimorphism in $\Modl{\mathcal S}$. Then we have the following short exact sequence in $\Modl{\mathcal S}$,
\begin{displaymath}
	0 \rightarrow \Ker f \rightarrow F \rightarrow G \rightarrow 0
\end{displaymath}
with $\Ker f$ and $G$ being $\tau$-cotorsion-free by Proposition \ref{p:CharacterizationFlats}. Then $F$ is $\tau$-cotorsion-free as well by Proposition \ref{p:PropertiesCotorsionFree}. This means that $F$ is flat in $\Flatl{\mathcal S}$, again by Proposition \ref{p:CharacterizationFlats}.

(3) First, notice that the class of flat objects in $\Flatl{\mathcal S}$ is closed under direct sums by propositions \ref{p:PropertiesCotorsionFree} and \ref{p:CharacterizationFlats}. Second, since every direct limit $L$ of flat objects in $\Flatl{\mathcal S}$ is a pure quotient of a direct sum of flat objects in $\Flatl{\mathcal S}$, $L$ is flat in $\Flatl{\mathcal S}$ by (1).

(4) This follows from Proposition \ref{p:CharacterizationFlats} and Corollary \ref{c:PureSubobjectsJansian}.
\end{proof}

Using Theorem \ref{t:Deconstructible}, we can obtain the similar result for flat objects in $\Flatl{\mathcal S}$ when pure submodules of flat objects are pure in $\Flatl{\mathcal S}$. As a consequence of Corollary \ref{c:EnoughFlatsImpliesPureSubmodules}, this result extends  \cite[Theorem 3]{Rump}.

\begin{theorem}\label{t:Rump}
	Assume that $\tau$ is hereditary and that the class of flat objects in $\Flatl{\mathcal S}$ is closed under pure subobjects. Then:
		\begin{enumerate}
		\item Every flat object in $\Flatl{\mathcal S}$ is the direct union of $<w(\mathcal S)^+$-presented (in $\Modl{\mathcal S}$) flat objects.
		
		\item Every flat object in $\Flatl{\mathcal S}$ is filtered by $w(\mathcal S)^+$-presented (in $\Modl{\mathcal S}$) flat objects. In particular, $\Flatl{\mathcal S}$ is deconstructible.
	\end{enumerate}
\end{theorem}

\begin{proof}
	Let $F$ be a flat object in $\Flatl{\mathcal S}$. By Theorem \ref{t:Deconstructible}(1), $F$ is the direct union of $<w(\mathcal S)^+$-presented pure (in $\Modl{\mathcal S}$) subobjects. By hypothesis, \ref{p:PropertiesFlats}, these pure submodules are flat objects in $\Flatl{\mathcal S}$ and (1) holds.
	
	\medskip
	
	Now, by Theorem \ref{t:Deconstructible}, $F$ has a filtration, $(F_\alpha \mid \alpha < \mu)$, by $<w(\mathcal S)^+$-presented flat modules; moreover, $F_\alpha$ can be taken to be pure in $F$. Then $F_{\alpha}$ and $F_{\alpha+1}/F_\alpha$ are flat objects in $\Flatl{\mathcal S}$ by Proposition \ref{p:PropertiesFlats}. Thus, (2) holds.
\end{proof}

Now we give a characterization of when all objects in $\Flatl{\mathcal S}$ are flat in $\Flatl{\mathcal S}$:

\begin{proposition}\label{p:AllObjectsAreFlat}
	The following assertions are equivalent:
	\begin{enumerate}
		\item Every object in $\Flatl{\mathcal S}$ is flat in $\Flatl{\mathcal S}$.
		
		\item Every epimorphism in $\Flatl{\mathcal S}$ is an epimorphism in $\Modl{\mathcal S}$.
		
		\item Every representable functor is flat in $\Flatl{\mathcal S}$.
		
		\item Every module is $\tau$-torsion free.
		
		\item Every simple module embeds in a flat module.
		
		\item If $M$ is a non-zero $<\kappa$-presented module for some infinite regular cardinal $\kappa>w(\mathcal S)$, there exists a non-zero $<\kappa$-presented flat module $F$ and a nonzero morphism $f:M \rightarrow F$.
	\end{enumerate}
\end{proposition}

\begin{proof}
	(1) $\Rightarrow$ (2). Trivial by Proposition \ref{p:CharacterizationFlats}.
	
	(2) $\Rightarrow$ (3). Again by Proposition \ref{p:CharacterizationFlats}.
	
	(3) $\Rightarrow$ (4). First notice that $\mathcal U_\tau(H_a)=\{H_a\}$ for every $a \in \mathcal S$. Then, $\{\mathcal U_\tau(H_a) \mid a \in \mathcal S\}$ is a Gabriel topology so that, by Proposition \ref{p:TopologiesAndTorsion}, $\tau$ is hereditary.
	
	Given any $T \in \mathcal T_\tau$ and an epimorphism $f:P \rightarrow T$ in $\Modl{\mathcal S}$, with $P$ a direct sum of representables. Now, take a morphism $g:Q \rightarrow P$ with $\Img g=\Ker f$ and $Q$ proyective in $\Modl{\mathcal S}$. Since $\Ker f$ is $\tau$-dense in $P$, $f$ is an epimorphism in $\Flatl{\mathcal S}$. By Proposition \ref{p:PropertiesFlats}, $P$ is flat in $\Flatl{\mathcal S}$, so that $f$ is an epimorphism in $\Modl{\mathcal S}$. This implies that $\Ker f=P$ and $T=0$.
	
	(4) $\Rightarrow$ (1). Trivial.
	
	(4) $\Rightarrow$ (5). If $S$ is simple, there exists a non-zero morphism $f:S \rightarrow F$ with $F$ flat. But this morphism has to be monic.
	
	(5) $\Rightarrow$ (3). Suppose that (3) is not true. Then there exists a representable functor $H_a$ which is not flat. By Proposition \ref{p:CharacterizationFlats}, $H_a$ contains a proper $\tau$-dense submodule $D$. Then, $D$ is contained in a maximal submodule $M$ \cite[p. 21]{Mitchell}. Then, $H_a/M$ is $\tau$-torsion, being a quotient of $H_a/D$. In particular, $H_a/M$ does not embed in a flat module.
	
	(4) $\Rightarrow$ (6). Given any nonzero $<\kappa$-presented module there exists a nonzero morphism $f:M \rightarrow F$ for some flat module. Now $f(M)$ is $<\kappa$-generated and by Proposition \ref{p:KappaPresented}, it is $<\kappa$-presented. By Theorem \ref{t:Purification}, there exists a $<\kappa$-presented pure submodule $K$ of $F$ such that $f(M) \leq K$. Since $K$ is flat, (6) holds.
	
	(6) $\Rightarrow$ (4). Trivial.
\end{proof}

As a consequence of the preceding result we get:

\begin{corollary}
	Suppose that $\Flatl{\mathcal S}$ is abelian. Then every flat module is flat in $\Flatl{\mathcal S}$ if and only if $\Flatl{\mathcal S} = \Modl{\mathcal S}$, i. e., $\Modl{\mathcal S}$ is regular in the sense of von Neumann.
\end{corollary}

\begin{proof}
	If $\Flatl{\mathcal S}$ is regular in the sense of von Neumann, the result is trivial. Conversely, if every module in $\Flatl{\mathcal S}$ is flat in $\Flatl{\mathcal S}$, then every module is $\tau$-torsion-free by the preceding result. By Corollary \ref{c:TorsionPanoramic}(4), the weak global dimension of $\Modl{\mathcal S}$ is less than or equal to $1$. But by (1) of the same corollary, this weak global dimension has to be $0$, i. e., $\Modl{\mathcal S}$ is regular in the sense of von Neumann.
\end{proof}

We conclude this section characterizing finitely accessible additive categories with enough flat objects. We say that \textit{there are enough flat (resp. projective) objects in $\Flatl{\mathcal S}$} if for every $F \in \Flatl{\mathcal S}$ there exists an epimorphism in $\Flatl{\mathcal S}$, $f:G \rightarrow F$, with $G$ flat (resp. projective) in $\Flatl{\mathcal S}$.

\begin{theorem}\label{t:ExistenceEnoughFlats}
	Suppose that $\tau$ is hereditary. The following assertions are equivalent:
	\begin{enumerate}
		\item $\Flatl{\mathcal S}$ has enough flat objects.
		
		\item $\tau$ is jansian and there exists an epimorphism in $\Modl{\mathcal S}$, $f^a:G_a \rightarrow L_\tau(H_a)$ with $G_a$ being flat (in $\Modl{\mathcal S}$) and $\tau$-cotorsion-free for every $a \in \mathcal S$.
		
		\item There exists a flat (in $\Modl{\mathcal S}$) pseudoprojective and $\tau$-cotorsion-free module $G$ such that $\mathcal T_\tau = \{X \in \Modl{\mathcal S} \mid \Hom(G,X)=0\}$.
	\end{enumerate}
\end{theorem}

\begin{proof}
	(1) $\Rightarrow$ (2). Let $a \in \mathcal S$ and an epimorphism in $\Flatl{\mathcal S}$, $f^a:G_a \rightarrow H_a$ with $G_a$ being flat in $\Flatl{\mathcal S}$. Since $G_a$ is $\tau$-cotorsion-free by Proposition \ref{p:CharacterizationFlats}, $\Img f \leq C_\tau(H_a)$ by Proposition \ref{p:PropertiesCotorsionFree}. But $\Img f$ is $\tau$-dense by Proposition \ref{p:CharacterizationEpimorphisms}, so that $L_\tau(H_a) \leq \Img f$. Then $L_\tau(H_a)=\Img f = C_\tau(H_a)$ and, by Lemma \ref{l:L=C}, $L_\tau(H_a)$ is $\tau$-dense. By Proposition   \ref{p:CharacterizationJansian}, $\tau$ is jansian. Finally, notice that $f^a$ induces an epimorphism in $\Modl{\mathcal S}$ from $G_a$ to $L_\tau(H_a)$.
	
	(2) $\Rightarrow$ (3). Let $P=\bigoplus_{a \in \mathcal S}L_\tau(H_a)$. By Proposition \ref{p:CharacterizationJansian}, $P$ is pseudoprojective, $\mathcal T_\tau=\{X \in \Modl{\mathcal S}\mid \Hom(P,X)=0\}$ and $\Gen(P)$ is the class of all $\tau$-cotorsion-free modules. Let $G=\bigoplus_{a \in \mathcal S}G_a$. Then $G$ is flat in $\Flatl{\mathcal S}$ by Proposition \ref{p:PropertiesFlats} and, since $G$ is $\tau$-cotorsion-free by Proposition \ref{p:CharacterizationFlats}, $G \in \Gen(P)$. But by (2), $P \in \Gen(G)$ as well, so that $\Gen(G)=\Gen(P)$. This implies that $\mathcal T_\tau=\{X \in \Modl{\mathcal S}\mid \Hom(P,X)=0\}=\{X \in \Modl{\mathcal S}\mid \Hom(F,X)=0\}$. Now, since $P$ is pseudoprojective  $\Gen(P)=\Gen(M)$ is closed under extensions and $\mathcal T_\tau$ is closed under quotients, which implies that $G$ is pseudoprojective as well by the same proposition.
\end{proof}

The following is a remarkable consequence:

\begin{corollary}\label{c:EnoughFlatsImpliesPureSubmodules}
	Suppose that $\tau$ is hereditary and that there are enough flat objects in $\Flatl{\mathcal S}$. Then the class of flat objects in $\Flatl{\mathcal S}$ is closed under pure submodules.
\end{corollary}

\begin{proof}
	Follows from the preceding result and Proposition \ref{p:PropertiesFlats}.
\end{proof}

We can obtain a similar characterization of finitely accessible additive categories having enough projective objects.

\begin{theorem}\label{t:ExistenceEnoughProjectives}
	Suppose that $\tau$ is hereditary. The following assertions are equivalent:
	\begin{enumerate}
		\item $\Flatl{\mathcal S}$ has enough projective objects.
		
		\item $\tau$ is jansian and there exists an epimorphism in $\Modl{\mathcal S}$, $f^a:P_a \rightarrow L_\tau(H_a)$ with $P_a$ being projective (in $\Modl{\mathcal S}$) and $\tau$-cotorsion-free for every $a \in \mathcal S$.
		
		\item There exists a projective (in $\Modl{\mathcal S}$) and $\tau$-cotorsion-free module $P$ such that $\mathcal T_\tau = \{X \in \Modl{\mathcal S} \mid \Hom(P,X)=0\}$.
	\end{enumerate}
\end{theorem}

Now we apply our results to give examples of finitely accessible additive categories with enough projective and flat objects.

\begin{corollary}\label{c:LocallyArtinian}
	Suppose that $\tau$ is hereditary. 
	\begin{enumerate}
		\item If $C_\tau(H_a)$ has a projective cover for every $a \in \mathcal S$, then $\Flatl{\mathcal S}$ has enough projective objects if and only if $\tau$ is jansian.
		
		\item If $\Modl{\mathcal S}$ is locally artinian, $\Flatl{\mathcal S}$ has enough projective objects.
	\end{enumerate}
\end{corollary}

\begin{proof}
	(1) In view of the preceding theorem, we have to prove that if $\tau$ is jansian, then there exists an epimorphism $f:P \rightarrow C_\tau(H_a)$ in $\Modl{\mathcal S}$ with $P$ being projective in $\Flatl{\mathcal S}$. But this epimorphism is, precisely, the projective cover of $C_\tau(H_a)$ as a consequence of Proposition \ref{p:PropertiesCotorsionFree}.
	
	(2) Given $a \in \mathcal S$, the set $\mathcal U_\tau(H_a)$ has a minimal element $U$. This minimal element is $\tau$-cotorsion-free, since any $\tau$-dense submodule $K$ of $U$ is $\tau$-dense in $H_a$, because $H_a/K$ is an extension of $U/K$ over $H_a/U$. Then $U=K$. This implies that $U \leq C_\tau(H_a) \leq C_\tau(H_a)$ and, in particular, that $L_\tau(H_a)$ is a $\tau$-dense submodule of $H_a$. By Proposition \ref{p:CharacterizationJansian}, $\tau$ is jansian. Now, the result follows from (1).
\end{proof}

\section{The particular case of the ring} \label{s:Ring}

We finish the paper illustrating our results in the particular case in which $\mathcal S$ consists on just one object, that is, for the category of flat modules over the ring $R$.

We use the plain terms \textit{torsion} and \textit{torsion-free} for the usual torsion and torsion-free modules over a ring (i. e., a module $M$ is \textit{torsion-free} if $rx \neq 0$ for any $x \in M$ and $r \in R$ which is not a left zero divisor; and it is \textit{torsion} if for any nonzero $x \in M$ we can find a $r \in R$ which is not a left zero divisor with $rx=0$. Notice that the pair of classes \textit{torsion} and \textit{torsion-free} do not form, in general, a torsion theory). Recall that every flat module is torsion free \cite[36.7]{Wisbauer}. The relationship of these modules with the torsion theory $\tau$ is:

\begin{lemma}
Every torsion module is $\tau$-torsion. If $R$ is commutative, every $\tau$-torsion-free module is torsion-free.
\end{lemma}

\begin{proof}
	Suppose that $M$ is a torsion module and let $f:M \rightarrow F$ is a morphism with $F$ flat. Given $m \in M$ there exists a non left zero divisor $r$ such that $rm=0$. Then $rf(m)=0$ which means that $f(m)=0$ since $F$ is torsion free \cite[36.7]{Wisbauer}. Then $M$ is $\tau$-torsion.
	
	Suppose that $R$ is commutative and that $F$ is a module which is not torsion-free. Then there exists a non-zero $x \in F$ and an element $r \in R$ which is not a left zero divisor such that $rx=0$. But then there is a nonzero morphism from $R/Rr$ to $F$ and $R/Rr$ is a torsion module because $ry=0$ for every $y \in R/Rr$. In particular, $R/Rr$ is $\tau$-torsion and, consequently, $F$ is not $\tau$-torsion-free.
\end{proof}

Now we compute the torsion theory $\tau$ for some particular rings:

\begin{examples}\label{e:ExamplesTau}
	\begin{enumerate}
		\item The equality $\mathcal F_\tau=\Flatl{R}$ holds if and only if $\Modr{R}$ is locally coherent and has weak dimension less than or equal to $1$, that is, if $R$ is right semihereditary by \cite[Theorem 4.1]{Chase}.
		
		\item Suppose that $R$ is a Prüfer domain or a right Bezout domain, that is, a domain in which every finitely generated right ideal is cyclic. Then a module $F$ is flat if and only if it is torsion-free  by \cite[Theorem 4.35]{Rotman} and \cite[36.7]{Wisbauer}. Since $R$ is right semihereditary in both cases, then the class $\mathcal F_\tau$ is equal to the class of all torsion-free modules.
		
		If $R$ is a Prüfer domain or a (not necessarily commutative) discrete valuation domain, then $\tau$ is hereditary. Given any $M \in \Flatl R$ and a nonzero element $x$ in the injective envelope $E(M)$ take $r \in R$ with $rx=0$. Using that $M$ is essential in $E(M)$, we can find $t \in R$ with $0 \neq tx \in M$. If $R$ is commutative, then $rtx=0$ which implies that $r=0$, as $M$ is torsion-free. If $R$ is a discrete valuation domain, there exist $s_1,s_2 \in R$ with $rt=s_1r=ts_2$ (see \cite[p. 18]{KrylovTuganbaev}). Then $rtx=s_1rx=0$ and, again, $r=0$. In any case, $E(M)$ is flat, which implies that $\tau$ is hereditary by Proposition \ref{p:PropertiesOfTau}.
		
		\item If $R$ is is a left \textit{left IF}, that is, a ring in which every injective module is flat, then $\mathcal F_\tau=\Modl{R}$ by Proposition \ref{p:PropertiesOfTau}. These rings were introduced in \cite{Colby} and \cite{Jain}. By Corollary \ref{c:TorsionPanoramic}, if $\Flatl{R}$ is abelian, then $R$ is left FTF.
		
		\item A ring is left FTF \cite{GomezTorrecillas} if the class $\mathcal F'$ of all modules isomorphic to a submodule of a flat module is the torsion-free class of a hereditary torsion theory. For these rings, $\mathcal F_\tau=\mathcal F'$ and $\tau$ is hereditary.
		
		\item Suppose that $R$ is a \textit{fir}, that is, every left and right ideal is free of unique rank. Then $R$ is right hereditary and $\mathcal F_\tau=\Flatl{R}$. Moreover, $\mathcal T_\tau$ is the class of all \textit{bound} modules which consists of all modules $T$ such that $\Hom_R(T,R)=0$ (see \cite[Corollary 5.1.4]{Cohn}).
		
		\item By \cite[Theorem 2.2]{Khash}, if $R$ is commutative Noetherian and MQF-3 (i.e., $E(\prescript{}{R}{R})$ is flat), then the injective hull of every flat module is flat. By Proposition \ref{p:PropertiesOfTau}, $\tau$ is hereditary.
		
		\item The torsion theory $\tau$ is not, in general, hereditary. Let $T$ be the triangular matrix ring $\mat{\mathbb Q}{0}{\mathbb Q}{\mathbb Z}$. The ring $S$ is right hereditary so that, by (1), $\mathcal F_\tau=\Flatl T$. Using the description of left $T$-modules as pairs $\vec{M_1}{M_2}_\varphi$, where $\varphi:M_1 \rightarrow M_2$ is a morphism of $\mathbb Z$-modules, it is easy to prove, using \cite[Corollary 3.4.9]{KrylovTuganbaev}, that the injective hull of ${_T}T$ is the module
		\begin{displaymath}
			E=\vec{\mathbb Q}{0}_{\varphi_1} \oplus \vec{\Hom_{\mathbb Z}(\mathbb Q,\mathbb Q \oplus \mathbb Q)}{\mathbb Q \oplus \mathbb Q}_{\varphi_2}
		\end{displaymath}
		where $\varphi_1$ is the zero morphism and $\varphi_2(f)=f(1)$ for each $f \in \Hom_{\mathbb Z}(\mathbb Q,\mathbb Q \oplus \mathbb Q)$. But $\vec{\mathbb Q}{0}_{\varphi_1}$ is not flat by \cite[Corollary 3.6.6]{KrylovTuganbaev}. Then, $T$ is not left MQF-3 and, in particular, $\mathcal F_\tau$ is not closed under injective envelopes. This means that $\tau$ is not hereditary by Proposition \ref{p:PropertiesOfTau}.
	\end{enumerate}
\end{examples}

Now we see that if $R$ is a domain, there are no non-trivial flat objects in $\Flatl{\mathcal S}$.

\begin{proposition}\label{e:FlatsInDomains}
	Let $R$ be a domain.
	\begin{enumerate}
		\item If $M$ is a module and $K \leq M$ is an essential submodule, then $K$ is $\tau$-dense.
		
		\item $\Flatl{R}$ does not have non-trivial flat objects.
	\end{enumerate}
\end{proposition}

\begin{proof}
	(1) Given a non-zero $x \in M-K$, there exists a nonzero $r \in R$ with $0 \neq rx \in K$. This means that $r(x+K)=0$ and that $M/K$ is torsion. By the preceding proposition, $M/K$ is $\tau$-torsion.
	
	(2) First, notice that every non-zero left ideal $I$ of $R$ is $\tau$-dense, since if $R/I$ is a torsion module. In particular, as $R$ is not simple, there are no simple flat modules. Now, let $F$ be a non-zero flat module. As $F$ is not semisimple, it contains a submodule $L$ which is not a direct summand. Now, the set $\{K \leq F \mid K \cap L=0\}$ is not empty and inductive, so that it contains a maximal element $K$ by Zorn's lemma. Since $L$ is not a direct summand, $K+L$ is a proper essential submodule of $F$. By (1), $K+L$ is a $\tau$-dense submodule of $F$; then, by Proposition \ref{p:CharacterizationFlats}, $F$ is not flat in $\Flatl{R}$. 
\end{proof}

For perfect rings we have:

\begin{corollary}
	Suppose that $\tau$ is hereditary and that $R$ is two-sided perfect. Then there are enough projectives in $\Flatl R$.
\end{corollary}

\begin{proof}
	Follows from Corollary \ref{c:LocallyArtinian}, since by \cite[Proposition 3.9]{Golan}, every hereditary torsion theory in $\Modr{R}$ is jansian.
\end{proof}

For noetherian serial rings we have:

\begin{corollary}
	Let $R$ be a noetherian serial ring. Then the following assertions are equivalent:
	\begin{enumerate}
		\item There are enough projectives in $\Flatl{R}$.
		
		\item $R$ is artinian.
	\end{enumerate}
\end{corollary}

\begin{proof}
	(1) $\Rightarrow$ (2). By \cite[Theorem 5.29]{Facchini}, $R$ is isomorphic to $S \times T$, where $S$ is an artinian serial ring and $T$, a hereditary noetherian serial ring with no simple left ideals. The hypothesis implies that there are enough projectives in $\Flatl{T}$. We see that there are no nonzero projectives in $\Flatl{R}$, which implies that $T$ has to be zero.
	
	Let $P$ be a nonzero projective module. Then it has a maximal submodule $K$. Since $R$ is left noetherian, $P/K$ is finitely presented. Now, $\Hom_R(P/K,R)=0$, as $R$ does not contain any simple left ideal, so that $P/K$ is torsion by Proposition \ref{p:PropertiesOfTau}. Then $K$ is a $\tau$-dense submodule of $P$ and it is not projective in $\Flatl{R}$ by Proposition \ref{p:CharacterizationFlats}.
	
	(2) $\Rightarrow$ (1). Follows from Corollary \ref{c:LocallyArtinian}.
\end{proof}

For local rings we have:

\begin{corollary}
	Suppose that $R$ is a local ring. Then::
	\begin{enumerate}
		\item If $\Soc({_R}R)\neq 0$, then every object in $\Flatl{R}$ is flat.
		
		\item If $\Soc({_R}R) = 0$, then there are no non-zero projectives in $\Flatl{R}$.
	\end{enumerate}
\end{corollary}

\begin{proof}
	(1) Notice that $R/J(R)$ embeds in ${_R}R$, so that, $R$ does not contain a proper $\tau$-dense left ideal (any $\tau$-dense proper left ideal $I$ would be contained in $J(R)$ and $R/J(R)$ would be $\tau$-torsion, being a quotient of $R/I$). In particular $R$ is projective in $\Flatl{R}$ and the result follows from Proposition \ref{p:AllObjectsAreFlat}.
	
	(2) In this case, $R/J(R)$ is $\tau$-torsion and $\mathcal U_\tau({_R}R)$ is not trivial. This implies that $\mathcal U_\tau(F)$ is not trivial for every free module $F$ and, since every projective module is free, we conclude that there are no non-zero projectives in $\Flatl{R}$.
\end{proof}

Using a recent result by Herbera and Prihoda we get, for commutative rings, the following:

\begin{corollary}
	Suppose that $R$ is commutative and that $\Flatl{R}$ has enough projective objects. Then every object in $\Flatl{R}$ is flat in $\Flatl{R}$.
\end{corollary}

\begin{proof}
	By Theorem \ref{t:ExistenceEnoughProjectives}, there exists a projective module $P$ such that $\mathcal T_\tau=\{X \in \Modl{R}\mid \Hom_R(P,X)=0\}$. By Proposition \ref{p:CharacterizationJansian}, $C_\tau(R)$ is equal to $\tr_P(R)$ so that it is a pure ideal of $R$ by \cite[Corollary 2.13]{HerberaPrihoda}. Since $C_\tau(R)$ is $\tau$-dense in $R$ by Theorem \ref{t:ExistenceEnoughProjectives}, we conclude that $C_\tau(R)=R$. In particular $R$ is flat in $\Flatl{R}$ and the result follows from Proposition \ref{p:AllObjectsAreFlat}.
\end{proof}

\section*{Acknowledgments}

The author would like to thank Juan Cuadra for several stimulating conversations on the topic of this article.

\bibliographystyle{plain} \bibliography{references}

\begin{thebibliography}{10}

\bibitem{AdamekRosicky}
Ji{\v{r}}{\'{\i}} Ad{\'a}mek and Ji{\v{r}}{\'{\i}} Rosick{\'y}.
\newblock {\em Locally presentable and accessible categories}, volume 189 of
  {\em London Mathematical Society Lecture Note Series}.
\newblock Cambridge University Press, Cambridge, 1994.

\bibitem{AndersonFuller}
Frank~W. Anderson and Kent~R. Fuller.
\newblock {\em Rings and categories of modules}, volume~13 of {\em Graduate
  Texts in Mathematics}.
\newblock Springer-Verlag, New York, second edition, 1992.

\bibitem{BicanBashirEnochs}
L.~Bican, R.~El~Bashir, and E.~Enochs.
\newblock All modules have flat covers.
\newblock {\em Bull. London Math. Soc.}, 33(4):385--390, 2001.

\bibitem{Buhler}
Theo B{\"u}hler.
\newblock Exact categories.
\newblock {\em Expo. Math.}, 28(1):1--69, 2010.

\bibitem{Chase}
Stephen~U. Chase.
\newblock Direct products of modules.
\newblock {\em Trans. Amer. Math. Soc.}, 97:457--473, 1960.

\bibitem{Cohn}
P.~M. Cohn.
\newblock {\em Free ideal rings and localization in general rings}, volume~3 of
  {\em New Mathematical Monographs}.
\newblock Cambridge University Press, Cambridge, 2006.

\bibitem{Colby}
R.~R. Colby.
\newblock Rings which have flat injective modules.
\newblock {\em J. Algebra}, 35:239--252, 1975.

\bibitem{CortesCriveiSaorin}
Manuel Cort\'es-Izurdiaga, Septimiu Crivei, and Manuel Saor\'in.
\newblock Reflective and coreflective subcategories.
\newblock {\em J. Pure Appl. Algebra}, 227(5):Paper No. 107267, 43, 2023.

\bibitem{Crawley}
William Crawley-Boevey.
\newblock Locally finitely presented additive categories.
\newblock {\em Comm. Algebra}, 22(5):1641--1674, 1994.

\bibitem{Crivei}
Septimiu Crivei.
\newblock Maximal exact structures on additive categories revisited.
\newblock {\em Math. Nachr.}, 285(4):440--446, 2012.

\bibitem{CriveiPrestTorrecillas}
Septimiu Crivei, Mike Prest, and Blas Torrecillas.
\newblock Covers in finitely accessible categories.
\newblock {\em Proc. Amer. Math. Soc.}, 138(4):1213--1221, 2010.

\bibitem{CuadraSimson}
Juan Cuadra and Daniel Simson.
\newblock Flat comodules and perfect coalgebras.
\newblock {\em Comm. Algebra}, 35(10):3164--3194, 2007.

\bibitem{EnochsJenda}
Edgar~E. Enochs and Overtoun M.~G. Jenda.
\newblock {\em Relative homological algebra}, volume~30 of {\em de Gruyter
  Expositions in Mathematics}.
\newblock Walter de Gruyter \& Co., Berlin, 2000.

\bibitem{Facchini}
Alberto Facchini.
\newblock {\em Module theory}.
\newblock Modern Birkh\"auser Classics. Birkh\"auser/Springer Basel AG, Basel,
  1998.
\newblock Endomorphism rings and direct sum decompositions in some classes of
  modules, [2012 reprint of the 1998 original] [MR1634015].

\bibitem{GarciaSimson}
J.~L. Garcia and D.~Simson.
\newblock On rings whose flat modules form a {G}rothendieck category.
\newblock {\em Colloq. Math.}, 73(1):115--141, 1997.

\bibitem{Golan}
Jonathan~S. Golan.
\newblock {\em Torsion theories}, volume~29 of {\em Pitman Monographs and
  Surveys in Pure and Applied Mathematics}.
\newblock Longman Scientific \& Technical, Harlow; John Wiley \& Sons, Inc.,
  New York, 1986.

\bibitem{GomezTorrecillas}
Jos\'{e} G\'{o}mez~Torrecillas and Blas Torrecillas.
\newblock Flat torsionfree modules and {QF}-{$3$} rings.
\newblock {\em Osaka J. Math.}, 30(3):529--542, 1993.

\bibitem{HerberaPrihoda}
Dolors Herbera and Pavel P\v~r\'ihoda.
\newblock Reconstructing projective modules from its trace ideal.
\newblock {\em J. Algebra}, 416:25--57, 2014.

\bibitem{CortesTorrecillas06}
M.~C. Izurdiaga and B.~Torrecillas.
\newblock Various classes of pseudoprojective modules over semiperfect rings.
\newblock {\em Comm. Algebra}, 34(3):797--815, 2006.

\bibitem{Jain}
Saroj Jain.
\newblock Flat and {FP}-injectivity.
\newblock {\em Proc. Amer. Math. Soc.}, 41:437--442, 1973.

\bibitem{Khash}
K.~Khashyarmanesh and Sh. Salarian.
\newblock On the rings whose injective hulls are flat.
\newblock {\em Proc. Amer. Math. Soc.}, 131(8):2329--2335, 2003.

\bibitem{KrylovTuganbaev}
Piotr Krylov and Askar Tuganbaev.
\newblock {\em Formal matrices}, volume~23 of {\em Algebra and Applications}.
\newblock Springer, Cham, 2017.

\bibitem{Mitchell}
Barry Mitchell.
\newblock Rings with several objects.
\newblock {\em Advances in Math.}, 8:1--161, 1972.

\bibitem{Mumford}
David Mumford.
\newblock {\em The red book of varieties and schemes}, volume 1358 of {\em
  Lecture Notes in Mathematics}.
\newblock Springer-Verlag, Berlin, expanded edition, 1999.
\newblock Includes the Michigan lectures (1974) on curves and their Jacobians,
  With contributions by Enrico Arbarello.

\bibitem{OberstRohrl}
Ulrich Oberst and Helmut R\"{o}hrl.
\newblock Flat and coherent functors.
\newblock {\em J. Algebra}, 14:91--105, 1970.

\bibitem{Popescu}
N.~Popescu.
\newblock {\em Abelian categories with applications to rings and modules},
  volume No. 3 of {\em London Mathematical Society Monographs}.
\newblock Academic Press, London-New York, 1973.

\bibitem{Prest}
Mike Prest.
\newblock {\em Purity, spectra and localisation}, volume 121 of {\em
  Encyclopedia of Mathematics and its Applications}.
\newblock Cambridge University Press, Cambridge, 2009.

\bibitem{Rotman}
Joseph~J. Rotman.
\newblock {\em An introduction to homological algebra}.
\newblock Universitext. Springer, New York, second edition, 2009.

\bibitem{Rump}
Wolfgang Rump.
\newblock Flat covers in abelian and in non-abelian categories.
\newblock {\em Adv. Math.}, 225(3):1589--1615, 2010.

\bibitem{Simson77}
Daniel Simson.
\newblock On pure global dimension of locally finitely presented {G}rothendieck
  categories.
\newblock {\em Fund. Math.}, 96(2):91--116, 1977.

\bibitem{Stenstrom68}
Bo~Stenstr\"{o}m.
\newblock Purity in functor categories.
\newblock {\em J. Algebra}, 8:352--361, 1968.

\bibitem{Stenstrom}
Bo~Stenstr{\"o}m.
\newblock {\em Rings of quotients}.
\newblock Springer-Verlag, New York, 1975.
\newblock Die Grundlehren der Mathematischen Wissenschaften, Band 217, An
  introduction to methods of ring theory.

\bibitem{Stovicek14}
Jan {\v{S}}{\v{t}}ov{\'{\i}}{\v{c}}ek.
\newblock On purity and applications to coderived and singularity categories.
\newblock Preprint, arXiv:, 2014.

\bibitem{Wisbauer}
Robert Wisbauer.
\newblock {\em Grundlagen der {M}odul- und {R}ingtheorie}.
\newblock Verlag Reinhard Fischer, Munich, 1988.
\newblock Ein Handbuch f{\"u}r Studium und Forschung. [A handbook for study and
  research].

\bibitem{Wisbauer96}
Robert Wisbauer.
\newblock On module classes closed under extensions.
\newblock In {\em Rings and radicals ({S}hijiazhuang, 1994)}, volume 346 of
  {\em Pitman Res. Notes Math. Ser.}, pages 73--97. Longman, Harlow, 1996.

\end{thebibliography}

\end{document}